\documentclass[a4paper, reqno, 12pt]{amsart}

\usepackage[utf8]{inputenc}
\usepackage{amsthm,amsfonts,amssymb,amsmath,amsxtra,amsrefs}
\usepackage{mathtools}
 \usepackage{bbold}
\usepackage{latexsym}
 \usepackage{stmaryrd}
\usepackage{graphicx}
\usepackage{tikz}
\usepackage{tikz-cd}
\usepackage{todonotes,cancel}
\usepackage[margin=1.25in]{geometry}
\usepackage{mathrsfs}
\usepackage{bm}
\usepackage[all]{xy}
\SelectTips{cm}{}
\usepackage{xr-hyper}
\usepackage[colorlinks=false,
   citecolor=Black,
   linkcolor=Red,
   urlcolor=Blue]{hyperref}
\usepackage{verbatim}
\RequirePackage{xspace}
\RequirePackage{etoolbox}
\RequirePackage{varwidth}
\RequirePackage{enumitem}
\RequirePackage{tensor}
\RequirePackage{mathtools}
\RequirePackage{longtable}
\RequirePackage{multirow}

\usepackage{rotating}

\tikzset{
    labl/.style={anchor=south, rotate=90, inner sep=.5mm}
}
\newtheorem{thm}{Theorem}
\newtheorem{theorem}[thm]{Theorem}
\newtheorem{thmintro}{Theorem}

\newtheorem{prop}[thm]{Proposition}

\newtheorem{lem}[thm]{Lemma}
\newtheorem{lemma}[thm]{Lemma}

\newtheorem{cor}[thm]{Corollary}

\theoremstyle{definition}
\newtheorem{definition}[thm]{Definition}
\newtheorem{defi}[thm]{Definition}

\newtheorem{rem}[thm]{Remark}
\newtheorem{remark}[thm]{Remark}

\numberwithin{equation}{section}
\numberwithin{thm}{section}

\newcommand{{\BG}}{\ensuremath{\mathbb {G}}\xspace}

\newcommand{{\BK}}{\ensuremath{\mathbb {K}}\xspace}

\newcommand{\BQ}{\ensuremath{\mathbb {Q}}\xspace}

\newcommand{\BZ}{\ensuremath{\mathbb {Z}}\xspace}

\newcommand{\CJ}{\ensuremath{\mathcal {J}}\xspace}

\newcommand{\CL}{\ensuremath{\mathcal {L}}\xspace}
\newcommand{\CM}{\ensuremath{\mathcal {M}}\xspace}

\newcommand{\CV}{\ensuremath{\mathcal {V}}\xspace}

\newcommand{\iwtp}{{\breve{X}^+}}
\newcommand{\iwt}{\breve{X}}
\newcommand{\Ir}{\mathbf{I}_\circ'}
\newcommand{\I}{\mathbf{I}}

\begin{document}

\title[]{Dual canonical bases and embeddings of symmetric spaces}

\author[Huanchen Bao]{Huanchen Bao}
\address{Department of Mathematics, National University of Singapore, Singapore.}
\email{huanchen@nus.edu.sg}

\author[Jinfeng Song]{Jinfeng Song}
\address{Department of Mathematics, National University of Singapore, Singapore.}
\email{j\_song@u.nus.edu}

 \subjclass[2020]{} 

\begin{abstract}
 For a connected reductive group $G_k$ over an algebraically closed field $k$ of char $\neq 2$ and a fixed point subgroup $K_k$ under an algebraic group involution, we construct a quantization and an integral model of any affine embeddings of the symmetric space $G_k/K_k$. We show that the coordinate ring of any affine embedding of $G_k/K_k$ admits a dual canonical basis. 
 
 We further construct an integral model for the canonical embedding (that is, an embedding which is complete, simple, and toroidal) of $G_k/K_k$. When $G_k$ is of adjoint type, we obtain an integral model for the wonderful compactification of the symmetric space.  
\end{abstract}

	\maketitle
	
%	\tableofcontents
	
\section{Introduction}
 
\subsection{} 
Let $G_k$ be a connected reductive group over an algebraically closed field $k$ of char $\neq 2$. Let $\theta_k$ be an involution of $G_k$ and denote the fixed point subgroup by $K_k$. The affine quotient $G_k/K_k$ is called a \emph{symmetric space}. In our previous paper \cite{BS}, we studied the coordinate ring $k[G_k/K_k]$ of the symmetric space. 
In the current paper, we study embeddings of symmetric spaces. An \emph{embedding} of $G_k/K_k$ is a normal $G_k$-variety $V_k$ with a $G_k$-equivariant open embedding $G_k/K_k\hookrightarrow V_k$. 

Let $\iwtp$ be the set of spherical dominant weights (see \S\ref{sec:iwt}). Let $\mathcal{L}$ be a finitely generated submonoid of $\iwtp$. Define  a subspace 
\begin{equation*}
    R_k(\mathcal{L})=\bigcup_{\mu\in \mathcal{L}}k[G_k/K_k]_{\leq\mu} \subset k[G_k/K_k].
\end{equation*}
 It is clear that $R_k(\mathcal{L})$ is moreover a $G_k$-subalgebra. Define the affine $G_k$-variety
\begin{equation}\label{eq:V}
V_k(\mathcal{L})=Spec\,R_k(\mathcal{L}).
\end{equation}
Then the embedding $R_k(\mathcal{L})\hookrightarrow k[G_k/K_k]$ of $G_k$-algebras defines a $G_k$-equivariant map $G_k/K_k \rightarrow V_k(\mathcal{L})$. 

The first main result of this paper is that \eqref{eq:V} gives an explicit construction of all  affine embeddings of symmetric spaces.

\begin{thmintro}\label{thm:emb} (Theorem~\ref{thm:semigroup})
    The map $\mathcal{L}\mapsto V_k(\mathcal{L})$ is a bijection between the set of closed saturated submonoids of $\iwtp$ with the set of affine embeddings of $G_k/K_k$ (up to isomorphisms).
\end{thmintro}

See Definition \ref{def:sat} for the definition of saturated submonoids.  
A key ingredient of the theorem is to establish the normality of $V_k(\mathcal{L})$, which relies essentially on the fact that $k[G_k/K_k]$ admits a good filtration, as established in \cite{BS}.

\subsection{}The classification of symmetric spaces is independent of the base field by Springer \cite{Sp87}, provided the characteristic $\neq 2$. Thanks to Theorem \ref{thm:emb}, the classification of their affine embeddings is again independent of the base field.
Our second main result is a construction of the integral models and quantization for the affine embeddings of symmetric spaces.

\begin{thmintro}\label{thm:ing}(Definition~\ref{def:caq} \& Theorem~\ref{thm:VL} \& Definition~\ref{defi:dualCBL}) Let $\CL$ be a saturated submonoid of $\iwtp$. We define a commutative ring $\mathbf{R}(\mathcal{L})$ and an affine scheme $\mathbf{V}(\mathcal{L})=Spec\; \mathbf{R}(\mathcal{L})$ such that
\begin{itemize}
    \item the affine scheme $\mathbf{V}(\mathcal{L})$ is an integral model of affine embeddings, that is, the geometric fibre $Spec\; k\times_{Spec\; \mathbb{Z}}\mathbf{V}(\mathcal{L})$ over any algebraically closed field $k$ of characteristic $\neq 2$ is isomorphic to the affine embedding $V_k(\mathcal{L})$;
    \item the commutative ring $\mathbf{R}(\mathcal{L})$ admits a natural $\BZ$-basis $\mathrm{B}(\mathcal{L)}$ which specializes to a basis $\mathrm{B}(\mathcal{L)}$ of the coordinate ring $k[V_k(\CL)]$; 
    \item there is a non-commutative $\mathbb{Z}[q,q^{-1}]$-algebra $\mathbf{R}_q(\mathcal{L})$ with a natural $\mathbb{Z}[q,q^{-1}]$-basis $\mathrm{B}_q(\mathcal{L})$, such that the base change $\mathbb{Z}\otimes_{q\mapsto 1}\mathbf{R}_q(\mathcal{L})$ is canonically isomorphic to $\mathbf{R}(\mathcal{L})$ and $\mathrm{B}_q(\mathcal{L})$ is mapped to $\mathrm{B}(\mathcal{L})$ under the isomorphism.
\end{itemize}
 We call $\mathrm{B}(\mathcal{L)}$ the dual canonical basis of the embedding $V_k(\CL)$. 
\end{thmintro}

We further establish various results for closures of $G_k$-orbits on such embeddings in Theorem~\ref{thm:CJ}, as well as abelianizations of affine embeddings on \S\ref{sec:ab}.

\subsection{}

For the rest of this introduction, let us assume that $G_k$ is semisimple. 

When $G_k$ is of adjoint type, De Concini--Procesi (in characteristric 0) and De Concini--Springer (in arbitary characteristic $\neq 2$) constructed a smooth complete embedding of $G_k/K_k$, such that it has the unique closed $G_k$-orbit (simple), and the closure of any $B_k$-stable divisor of $P_k$ which is not $G_k$-stable does not contain the closed $G_k$-orbit (toroidal). This embedding is called the \emph{wonderful compactification of $G_k/K_k$}.

For general semisimple group $G_k$, one cannot expect a smooth complete simple toroidal embedding. However, one can still construct a (not necessarily smooth) complete simple toroidal embedding, which is unique up to isomorphism ([cf, \cite{Gan}]). We call such an embedding \emph{the canonical embedding of $G_k/K_k$} following Gandini \cite{Gan}*{Definition 11.4}.

Our third main result is a construction of the integral model for the canonical embeddings.

\begin{thmintro}(Theorem~\ref{thm:sp} \& Definition~\ref{defi:P})
    There exists a projective scheme $\mathbf{P}$ over $\mathbb{Z}$, such that its geometric fibre $Spec
    \, k\times_{Spec\,\mathbb{Z}}\mathbf{P}$ over any algebraically closed field $k$ of characteristic $\neq 2$ is isomorphic to the canonical embedding of $G_k/K_k$.
\end{thmintro}

 We study the local structure of the canonical embedding (Proposition \ref{thm:smoo}) and further provide a criterion for the smoothness of canonical embeddings. This includes the case when $G_k$ is adjoint, but also includes other cases. Let us mention that the local structure theorem was previously known for toroidal embeddings in characteristic zero by Brion--Luna--Vust \cite{BLV}, and was known for wonderful compactification in arbitrary characteristic by De Concini--Procesi--Springer \cites{DCP,DS}. See also the work of Tange \cite{Tange} for a general discussion of the local structure theorem for spherical embeddings.

    \vspace{.2cm}
\noindent {\bf Acknowledgment: } Both authors are supported by MOE grants A-0004586-00-00 and A-0004586-01-00. We thank Tsao-Hsien Chen for helpful discussion. 
 \section{Preliminaries}

In this section, we recall some constructions of symmetric spaces. Results in Section \ref{sec:scs} and Proposition \ref{prop:va} are new.

We retain the same notation as in the previous section.

\subsection{The $\imath$root datum}\label{sec:iwt}
Let $k$ be an algebraically closed field of char $\neq 2$. Following Springer \cite{Sp87}*{1.1}, a torus $S_k$ of $G_k$ is called \emph{split} if $\theta_k(s)=s^{-1}$ for any $s\in S_k$. Let $T_k$ be a $\theta_k$-stable maximal torus of $G_k$, which contains a maximal split torus. Let $B_k$ be a Borel subgroup of $G_k$ containing $T_k$, such that $\theta_k(B_k)\cap B_k$ has the minimal dimension. The existence of the such a pair $(T_k,B_k)$ was proved in \cite{Sp87}*{1.4}. Let $U_k$ be the unipotent radical of $B_k$. Let $X$ be the group of characters on $T_k$. We shall view $X$ also as the character group of $B_k$ in the natural way. Let $Y$ be the group of cocharacters on $T_k$. We write $\langle\;,\;\rangle$ to denote the canonical pairing between $Y$ and $X$. Let $\{\alpha_i\}_{i\in \I}$ be the set of simple roots and $\{\alpha_i^\vee\}_{i\in\I}$ be the set of simple coroots. For $\lambda,\mu\in X$, we write $\lambda\leq \mu$ if and only if $\mu-\lambda$ is a non-negative linear combination of simple roots. Set $X^+=\{\mu\in X\mid \langle \alpha_i^\vee,\mu\rangle\geq 0,\text{ for any }i\in\I\}$ to be the set of \emph{dominant weights}. Let $Q=\mathbb{Z}\mathbb[\alpha_i\mid i\in\I]$ be the root lattice, and $E=\mathbb{R}\otimes_{\mathbb{Z}} Q$ be the associated real vector space. We equip $E$ with a Euclidean structure $(\cdot,\cdot)$, such that $\langle \alpha_i^\vee,\alpha_j\rangle=\frac{(\alpha_i,\alpha_j)}{2(\alpha_i,\alpha_i)}$ for any $i,j$ in $\I$. Then $\{\alpha_i\mid i\in\I\}$ forms a set of simple roots of a root system in $E$. Let $W=\langle s_i\mid i\in\I\rangle$ be the Weyl group associated with the root system. 

Since $\theta_k$ leaves $T_k$ stable, it induces involutions $\theta_X$ and $\theta_Y$ on the lattices $X$ and $Y$, respectively. It is known \cite{Sp87} that there is a subset $\I_\bullet\subset\I$ and an involution $\tau:\I\rightarrow\I$ preserving $\I_\bullet$, such that: (i) $w_\bullet \alpha_i=-\alpha_{\tau i}$ for any $i\in\I_\bullet$, and (ii) $\theta_X(\alpha_i)=-w_\bullet \alpha_{\tau i}$ for any $i\in\I$. Here $w_\bullet$ is the longest element in the parabolic subgroup $W_{\I_\bullet}\subset W$ associated with the subset $\I_\bullet$. Write $\I_\circ=\I-\I_\bullet$. Then $\tau(\I_\circ)=\I_\circ$. Let us fix a set $\Ir\subset\I_\circ$ of representatives of $\tau$-orbits on $\I_\circ$. The tuple $(\I=\I_\bullet\sqcup\I_\circ,\tau,Y,X,\langle\;,\;\rangle,\theta_X,\theta_Y)$ is called an \emph{$\imath$root datum} associated with $(G_k,\theta_k)$. The pair $(G_k,\theta_k)$ is completely determined by its $\imath$root datum up to isomorphisms.

Following \cite{BW18}*{(3.3)}, we write
\begin{equation*}
    \breve{X}=\{\mu-\theta_X(\mu)\mid\mu\in X\}
\end{equation*}
to be the sublattice of $X$. We set $\iwtp=\iwt\cap X^+$. For $\lambda\in X$, let us write $\overline{\lambda}=\lambda-\theta_X(\lambda)$. Let us identify $\iwt$ with the character lattice of the torus $\overline{T}_k = T_k/T_k^{\theta_k}$ in the canonical way.

\subsection{The spherical weight lattice}\label{sec:scs}

We study the structure of the lattice $\iwt$ in this section.  

\begin{lemma}\label{le:bar}
    If $\lambda\in X^+$, then $\overline{\lambda} \in \iwtp$.
\end{lemma}

\begin{proof}
    It follows from the definition that $\overline{\lambda}\in \iwt$. It suffices to show that $\overline{\lambda}\in X^+$. For $i\in\I$, we have
    $
    \langle \alpha^\vee_i,\overline{\lambda}\rangle=\langle \alpha_i^\vee-\theta_Y(\alpha_i^\vee),\lambda\rangle.   
    $
    Note that $\theta_Y(\alpha_i^\vee)=\alpha_i^\vee$ if $i\in\I_\bullet$, and $\theta_Y(\alpha_i^\vee)=-w_\bullet\alpha_{\tau i}^\vee$ which is a negative coroot. Hence $\alpha_i^\vee-\theta_Y(\alpha_i^\vee)$ is either zero or a sum of positive coroot. Therefore $\langle\alpha_i^\vee,\overline{\lambda}\rangle$ is nonnegative. We complete the proof.
\end{proof}

\begin{lemma}\label{le:dec}
    The lattice $\iwt$ is generated by $\iwtp$ as an additive group, that is, for any $\mu\in\iwt$, there are $\mu'$ and $\mu''$ in $\iwtp$, such that $\mu=\mu'-\mu''$.
\end{lemma}

\begin{proof}
    By definition we have $\mu=\overline{\lambda}$, for some $\lambda\in X$. It is clear that we can write $\lambda=\lambda'-\lambda''$, for $\lambda' $ and $\lambda''$ in $X^+$. Then $\mu=\overline{\lambda'}-\overline{\lambda''}$, where $\overline{\lambda'}$ and $\overline{\lambda''}$ belong to $\iwtp$ by Lemma \ref{le:bar}. 
\end{proof}

\begin{lemma}\label{le:min}
    Suppose $\lambda$ and $\mu$ belong to $\iwt$, and $\lambda\le\mu$. Then $2(\mu-\lambda)$ is a sum of elements in $\{\overline{\alpha_i}\mid i\in\Ir\}$. 
\end{lemma}

\begin{proof}
    Since $\lambda\le\mu$, let us write $\mu-\lambda=\sum_{i\in\I}n_i\alpha_i$, where $n_i$ are nonnegative integers. We claim that
    \begin{equation}\label{eq:2mu}
        2(\mu-\lambda)=\sum_{i\in\Ir,\;\tau i=i}n_i\overline{\alpha_i}+2\sum_{i\in\Ir,\;\tau i\neq i}n_i\overline{\alpha_i}.
    \end{equation}
    
    It follows from the definition that for $i\in \I_\bullet$, we have $\theta_X(\alpha_i)=\alpha_i$, and for $i\in\I_\circ$, we have
    \begin{equation}\label{eq:tx}
    \theta_X(\alpha_i)=-\alpha_{\tau i}-\sum_{j\in \I_\bullet}t_{ij}\alpha_j, \text{ where $t_{ij}$ are nonnegative integers.}
    \end{equation}

Since $\mu-\lambda$ belongs to $\iwt$, we have $\theta_X(\sum_{i\in\I}n_i\alpha_i)=-\sum_{i\in\I}n_i\alpha_i$. By \eqref{eq:tx}, we have 
\begin{align*}   \sum_{i\in\I}n_i\theta_X(\alpha_i)&=\sum_{i\in\I_\bullet}n_i\alpha_i-\sum_{i\in\I_\circ}n_i\big(\alpha_{\tau i}+\sum_{j\in\I_\bullet}t_{ij}\alpha_j\big)\\
    &=-\sum_{i\in\I_\circ}n_i\alpha_{\tau i}+\sum_{j\in\I_\bullet}\big(n_j-\sum_{i\in\I_\circ}n_it_{ij}\big)\alpha_j.
\end{align*}

Therefore we deduce that $n_{\tau i}=n_i$   for $i\in\I_\circ$, and  $2n_j=\sum_{i\in\I_\circ}n_it_{ij}$   for $j\in\I_\bullet$.

By \eqref{eq:2mu}, we have
\begin{align*}
    &\sum_{i\in\Ir,\;\tau i=i}n_i\overline{\alpha_i}+2\sum_{i\in\Ir,\;\tau i\neq i}n_i\overline{\alpha_i}\\=&\sum_{i\in\Ir,\;\tau i=i}n_i\big(2\alpha_i+\sum_{j\in\I_\bullet}t_{ij}\alpha_j\big)+2\sum_{i\in\Ir,\;\tau i\neq i}n_i\big(\alpha_i+\alpha_{\tau i}+\sum_{j\in\I_\bullet}t_{ij}\alpha_j\big)\\=&2\sum_{i\in\I_\circ}n_i\alpha_i+\sum_{j\in\I_\bullet}\big(\sum_{i\in\Ir,\;\tau i=i}n_it_{ij}+\sum_{i\in\Ir,\;\tau i\neq i}2n_it_{ij}\big)\alpha_j\\\stackrel{(\heartsuit)}{=}&2\sum_{i\in\I_\circ}n_i\alpha_i+\sum_{j\in\I_\bullet}\big(\sum_{i\in\I_\circ}n_it_{ij}\big)\alpha_j\\=&2\sum_{i\in\I_\circ}n_i\alpha_i+2\sum_{j\in\I_\bullet}n_j\alpha_j\\=&2\sum_{i\in\I}n_i\alpha_i.
\end{align*}
Here $(\heartsuit)$ follows from the fact that $t_{ij}=t_{\tau i,j}$, for $i\in\I_\circ$ and $j\in\I_\bullet$. Hence we proved the equality \eqref{eq:2mu}, which completes the proof of the lemma.
\end{proof}

\subsection{Spherical root system}\label{subsec:spherical}It is clear the set $\{\overline{\alpha_i}\mid i\in\Ir\}$ is linearly independent in the Euclidean space $E$. Let $\breve{E}\subset E$ be the subspace spanned by this set. For each $i\in\Ir$, set $\alpha_i'$ to be the generator of the semigroup $\mathbb{R}_+\overline{\alpha_i}\cap \iwt$. We call elements $\alpha_i'$ the \emph{spherical roots}. By \cite{SV}*{the discussion afetr Remark 2.1.1}, the set $\{\alpha_i'\mid i\in\Ir\}$ forms a set of simple roots of a root system in the Euclidean space $\breve{E}$, which is call the \emph{spherical root system}. Note that although \cite{SV} assumes the characteristic of $k$ is zero, the set $\{\alpha_i'\mid i\in\Ir\}\subset \iwt$ is independent of the base field $k$. 

\subsection{The Borel eigenvalues and eigenfunctions} \label{sec:BE}

For any $k$-linear space $V$ where $B_k$ acts linearly and any $\mu\in X$, set
\begin{equation}\label{eq:Vm}
V^{(\mu)}=\{v\in V\mid b\cdot v=\mu(b)v,\text{ for any }b\in B_k\}
\end{equation}
to be the subspace of $B_k$-eignfunctions associated with the character $\mu$. We set
\begin{equation}
    V^{U_k}=\{v\in V\mid u\cdot v=v,\text{ for any }u\in U_k\}
\end{equation}
to be the subspace of $U_k$-invariants. When $V$ is finite-dimensional, it is clear that $V^{U_k}$ is a direct sum of the subspaces $V^{(\mu)}$ for various $\mu\in X$.

The coordinate ring $k[G_k/K_k]$ admits a $G_k$-action in the natural way: for $g\in G_k$, $h\in G_k/K_k$, and $f\in k[G_k/K_k]$, we have $(g\cdot f)(h)=f(g^{-1}h)$. It follows by \cite{BS}*{Theorem 2} that $k[G_k/K_k]^{(\mu)}$ is not a zero space if and only if $\mu\in\iwtp$, and in which case the space $k[G_k/K_k]^{(\mu)}$ is 1-dimensional. 

For $\mu\in\iwtp$, we take $\chi_\mu\in k[G_k/K_k]^{(\mu)}$ such that $\chi_\mu(eK_k)=1$. Such function exists and is unique since $k[G_k/K_k]^{(\mu)}$ is 1-dimensional and $B_k\cdot eK_k\subset G_k/K_k$ is an open subset. 

Let $k(G_k/K_k)$ be the field of rational functions on $G_k/K_k$. For $\mu\in\iwt$, by Lemma \ref{le:dec} we can find $\mu'$, $\mu''$ in $\iwtp$, such that $\mu=\mu'-\mu''$, and define $\chi_{\mu}=\chi_{\mu'}/\chi_{\mu''}$ in $k(G_k/K_k)^{(\mu)}$. It is clear that $\chi_\mu$ is independent of the choice of $\mu'$ and $\mu''$. We remark that $\chi_\mu\cdot\chi_\lambda=\chi_{\mu+\lambda}$, for $\mu,\lambda\in \iwt$.

\subsection{The quantization, integral model, and good filtration}\label{sec:int}
We summarize some results on symmetric spaces in  \cite{BS}  to be used in this paper.

\begin{thmintro}\cite{BS}\label{thm:int}
    Given an $\imath$root datum, there is a non-commutative $\mathbb{Z}[q,q^{-1}]$-algebra $\mathbf{O}_q(G/K)$, and a (dual canonical) basis $B_q(G/K)\subset \mathbf{O}_q(G/K)$ as a free $\mathbb{Z}[q,q^{-1}]$-module. For any $\mu\in\iwtp$, there is a finite subset $B_q(G/K)_{\le\mu}$ of $B_q(G/K)$, such that $B_q(G/K)_{\le \mu}\subset B_q(G/K)_{\le \lambda}$ if $\mu\le\lambda$, and $B_q(G/K)$ is a union of $B_q(G/K)_{\le\mu}$ for all $\mu\in\iwtp$. 

    Let $\mathbf{O}(G/K)=\mathbb{Z}\otimes_{q\mapsto 1}\mathbf{O}_q(G/K)$, and $B(G/K)$ be the image of $B_q(G/K)$ under the base change. Then $\mathbf{O}(G/K)$ is a commutative ring, with the (dual canonical) basis $B(G/K)\subset \mathbf{O}(G/K)$ as a free $\mathbb{Z}$-module. For any $\mu\in\iwtp$, let $B(G/K)_{\le \mu}\subset B(G/K)$ be the image of $B_q(G/K)_{\le \mu}$ under the base change.
    
    For any algebraically closed field $k$ with characteristic not 2, the following properties hold:

    (1) The base change $k\otimes_\mathbb{Z}\mathbf{O}(G/K)$ is canonically isomorphic to the coordinate ring $k[G_k/K_k]$ of the symmetric space of the given type as $k$-algebras.

    (2) The dual canonical basis $B(G/K)$ specializes to a dual canonical basis of $k[G_k/K_k]$.

    (3) For any $\mu\in\iwtp$, let $\mathbf{O}(G/K)_{\le \mu}$ be the $\mathbb{Z}$-submodule spanned by $B(G/K)_{\le\mu}$, and let $k[G_k/K_k]_{\le\mu}\cong k\otimes_\mathbb{Z}\mathbf{O}(G/K)_{\le\mu}$ be the subspace of $k[G_k/K_k]$ under the isomorphism in (1). Then $k[G_k/K_k]_{\le\mu}$ is a $G_k$-submodule. Moreover the quotient module $k[G_k/K_k]_{\le\mu}/k[G_k/K_k]_{<\mu}$ is isomorphic to the dual Weyl module with the highest weight $\mu$. Here $k[G_k/K_k]_{<\mu}$ is the union of $k[G_k/K_k]_{\le\lambda}$, for all $\lambda\in\iwtp$ with $\lambda<\mu$.

    (4) For $\mu',\mu''\in \iwtp$, one has $k[G_k/K_k]_{\le\mu'}\cdot k[G_k/K_k]_{\le\mu''}\subset k[G_k/K_k]_{\mu'+\mu''}$.
\end{thmintro}

Recall the rational functions $\chi_\mu$ on $G_k/K_k$, for $\mu\in\iwt$. The following lemma is straightforward from \cite{BS}.
\begin{lemma}\label{le:Uk}
    For $\lambda,\mu\in\iwtp$, the element $\chi_\mu$ belongs to $k[G_k/K_k]_{\le\lambda}$ if and only if $\mu\le\lambda$. Moreover the space $k[G_k/K_k]_{\le\lambda}^{U_k}$ is spanned by various $\chi_\mu$ for $\mu\le \lambda$ as a $k$-vector space. 
\end{lemma}

Let $k(G_k/K_k)$ be the field of rational functions on $G_k/K_k$.
\begin{lem}\label{le:Lambda}
The set of $B_k$-eigenfunctions on $k(G_k/K_k)$ are  $\{\chi_{\mu} \vert \mu \in   \breve{X} \}$, up to multiplication by scalars.
\end{lem}
\begin{proof}
It is clear that $\{\chi_{\mu} \vert \mu \in   \breve{X} \}$ are $B_k$-eigenfunctions (see \cite{Gan}*{Theorem~2.8}). 

It follows from \cite{BS}*{Theorem 2} (see also Lemma~ \ref{le:Uk}) and Lemma~\ref{le:dec} that $\chi_{\mu}$ is an $B_k$-eigenfunction in $k(G_k/K_k)$ for any $\mu \in \breve{X}$. 
\end{proof}
It follows that $\breve{X}$ is the weight lattice of $G_k / K_k$ considered in \cite{Gan}*{Definition~2.13}.

\subsection{The theory of spherical embeddings}

An embedding $V_k$ of $G_k/K_k$ is called \emph{simple} if there is a unique closed $G_k$-orbit in $V_k$. This includes all the affine embeddings \cite{LV}*{Theorem 6.7}. All the embeddings in this paper will be simple. For a simple embedding $V_k$, a $B_k$-stable prime divisor $D$ of $V_k$ which is not $G_k$-stable is called a \emph{color} if $D$ contains the closed $G_k$-orbit. A simple embedding is called \emph{toroidal} if it has no colors. A complete toroidal simple embedding of $G_k/K_k$ is called a \emph{canonical embedding} of $G_k/K_k$. 

A \emph{valuation} on $k(G_k/K_k)$ is a map
\[
v:k(G_k/K_k)^*=k(G_k/K_k)\backslash\{0\}\longrightarrow \mathbb{Q}
\]
such that
\begin{itemize}
    \item[(1)] $v(f_1+f_2)\ge \text{min}\{v(f_1),v(f_2)\}$ for $f_1,f_2\in k(G_k/K_k)^*$;
    \item[(2)] $v(f_1f_2)=v(f_1)+v(f_2)$, for $f_1,f_2\in k(G_k/K_k)$;
    \item[(3)] $v(k^*)=0.$
\end{itemize}
The field $k(G_k/K_k)$ admits a natural $G_k$-action. A valuation $v$ is called \emph{$G_k$-invariant} if $v(g\cdot f)=v(f)$, for any $g\in G_k$ and any $f\in k(G_k/K_k)^*$.  We denote by $\CV=\CV(G_k/K_k)$ the set of $G_k$-invariant valuations on $G_k/K_k$. 

By Lemma \ref{le:Lambda} $\iwt$ is the lattice consisting of eignvalues of $B_k$-semi-invariant rational functions on $G_k/K_k$. Let $_\mathbb{Q}\iwt^*=\text{Hom}_{\mathbb{Z}}(\iwt,\mathbb{Q})$. Any valuation $v$ on $k(G_k/K_k)$ determines an element $\varrho(v)$ in $_\mathbb{Q}\iwt^*$, where $\varrho(v)(\mu)=v(\chi_\mu)$ for $\mu\in\iwt$. It is known that the map $\varrho$ is injective when restricting to the set $\CV$ (cf. \cite{LV}*{Corollary~7.8}). Let us identify $\CV$ with its image in $_\mathbb{Q}\iwt^*$. 

For a simple embedding $V_k$ of $G_k/K_k$, the associated \emph{colored cone} $(\mathcal{C},\mathcal{F})$ is the pair where $\mathcal{F}$ is the set of colors and $\mathcal{C}\subset {_\mathbb{Q}\iwt^*}$ is the cone generated by the images of $G_k$-stable prime divisors of $V_k$ and elements in $\mathcal{F}$ under the map $\varrho$. A simple embedding is completely determined by its colored cone. By \cite[Theorem 4.2]{LV}, a canonical embedding corresponds to the colored cone $(\mathcal{V},\emptyset)$. In particular, a canonical embedding is unique if it exists.

\subsection{The wonderful compactifications}\label{sec:wond}

We assume that $G_k$ is adjoint in this subsection. A construction of the canonical embedding $P_k$ of $G_k/K_k$ was given by De Concini--Procesi in characteristic 0 and De Concini--Springer in positive characteristic $\neq 2$. They further showed that $P_k$ is smooth and call it the \emph{wonderful compactification} of $G_k/K_k$. One important tool to analyze the geometry of $P_k$ is the local structure theorem which we now recall.

Let $\mathring{P}_k\subset P_k$ be the complement  of all the $B_k$-stable divisors which are not $G_k$-stable. Let $\overline{T_k}=T_k/T_k^{\theta_k}$ be the torus contained in $G_k/K_k$. Let $N_k$ be the closure of $\overline{T_k}$ in $\mathring{P}_k$. Let $P_{\I_\bullet, k} \supset B_k$ be the parabolic subgroup associated with the subset of simple roots $\I_\bullet$, and let $U_{P_{\I_\bullet, k}}$ be the unipotent radical of $P_{\I_\bullet, k}$. Recall that $\iwt$ is identified with the character lattice of $\overline{T_k}$.

\begin{theorem}[\cite{DS}*{Proposition 3.8}]\label{thm:smooad}
The action map
    \[
    U_{P_{\I_\bullet, k}}\times N_k\longrightarrow \mathring{P_k}
    \]
    is an isomorphism of varieties. And $N_k\supset \overline{T_k}$ is an affine toric variety whose coordinate ring is the monoid algebra of $C=\mathbb{N}[-\overline{\alpha_i}\mid i\in\Ir] \subset \iwt$. 
\end{theorem}

\subsection{The valuation cone}

We drop the assumption that $G_k$ is adjoint now. Recall that $\CV=\CV(G_k/K_k)$ is the cone consisting of $G_k$-invariant valuations on $k(G_k/K_k)$, which is identified with a subset of $_\mathbb{Q}\iwt^*=\text{Hom}_{\mathbb{Z}}(\iwt,\mathbb{Q})$. The following proposition is a consequence of the local structure theorem. 

\begin{prop}\label{prop:va}
We have 
    \begin{equation}\label{eq:va}
    \CV=\{t\in{_\mathbb{Q}\iwt^*}\mid t(\overline{\alpha_i})\le 0,\;\text{ for any }i\in\Ir\}\subset {_\mathbb{Q}\iwt^*}.
    \end{equation}
\end{prop}

\begin{proof}
By \cite{LV}*{Corollary~5.3}, the valuation cone always contains the right hand side of \eqref{eq:va}. It suffices to show that 
\begin{equation}\label{eq:inc}
\CV\subset\{t\in{_\mathbb{Q}\iwt^*}\mid t(\overline{\alpha_i})\le 0,\;\text{ for any }i\in\Ir\}.
\end{equation}

Firstly suppose $G_k$ is of adjoint type. Let $P_k$ be the wonderful compactification of $G_k/K_k$ as in Section \ref{sec:wond}. Since $P_k$ is canonical, the valuation cone $\CV$ is generated by the valuations associated with $G_k$-stable prime divisors of $P_k$. Let $D$ be a $G_k$-stable prime divisor of $P_k$, and $v_D$ be the valuation associated with $D$. Since $P_k$ has no colors, the intersection $D\cap \mathring{P_k}$ is non-empty, hence it is a prime divisor of the affine space $\mathring{P_k}$. By Theorem \ref{thm:smooad}, the functions $\chi_{-\overline{\alpha_i}}$, for $i\in\Ir$, are regular on $\mathring{P_k}$, so we have $v_D(\chi_{-\overline{\alpha_i}})\ge 0$. Hence \eqref{eq:inc} is proved.

Next suppose that $G_k$ is a general reductive group, and $Z_k$ is the center of $G_k$. Let $G_{k}^{ad}=G_k/Z_k$. The involution $\theta_k$ induces an involution on $G_{k}^{ad}$. Let $K_{k}^{ad}$ be the fixed point subgroup of $G_{k}^{ad}$. Then any $G_k$-invariant valuation on $k(G_k/K_k)$ induces a $G_{k}^{ad}$-invariant valuation on $k(G_{k}^{ad}/K_{k}^{ad})$. One may identify $k(G_{k}^{ad}/K_{k}^{ad})$ as a subfield of $k(G_k/K_k)$, and it is clear that $\chi_{\overline{\alpha_i}}$ belongs to $k(G_{k}^{ad}/K_{k}^{ad})$, for $i\in\Ir$. Then the proof follows from the adjoint case.
\end{proof}

\begin{remark}
    Proposition \ref{prop:va} is proved in \cite{Vu90} (see also \cite{LV}*{\S 5}) when the characteristic of $k$ is 0.
\end{remark}

\section{Affine embeddings of symmetric spaces}\label{sec:aff}

\subsection{Semigroups and affine embeddings} We firstly define several properties of subsemigroups of $\breve{X}^+$. A subsemigroup of $\breve{X}^+$ is a subset which is closed under addition. A submonoid is a subsemigroup with zero.

\begin{definition}\label{def:sat}
    A subsemigroup $\mathcal{L}$ of  $\iwtp$ is {\em saturated} if 
\begin{itemize}
\item $\CL$ is finitely generated as a semigroup;
\item $\CL$ generates $\breve{X}$ as an additive group;
 \item if $n\mu\in\mathcal{L}$ for some $n\in\mathbb{Z}_{>0}$ and $\mu\in \iwtp$, then $\mu\in\mathcal{L}$.
\end{itemize}
A subsemigroup $\CL$ of  $\breve{X}^+$ is called {\em closed} if for any $\lambda\in\iwtp$ with $\lambda\leq \mu$ for some $\mu \in \CL$, we have  $\lambda\in\CL$.
\end{definition}

For any subsemigroup $\CL$ of $\breve{X}^+$, we define the $k$-subspace 
\begin{equation}\label{eq:R}
R_k(\mathcal{L})=\bigcup_{\mu\in \mathcal{L}}k[G_k/K_k]_{\leq\mu}.
\end{equation}
If $\mathcal{L}$ is a submonoid, then it is clear that $R_k(\CL)$ is a unital $k$-algebra.  
\begin{lem}
Let $\mathcal{L}\subset\breve{X}^+$ be a closed saturated submonoid. Then $R_k(\mathcal{L})$ is a finitely-generated normal (integrally closed) unital $k$-algebra.
\end{lem}
\begin{proof}
It is clear that $R_k(\mathcal{L})$ inherits the good filtration from $k[G_k/K_k]$ in Theorem \ref{thm:int}. By \cite{Gross}*{Theorem 9}, it suffices to show that the $k$-algebra $R_k(\mathcal{L})^{U_k}$ of $U_k$-invariants is finitely generated and normal. 

     Let $k[\mathcal{L}]$ be the monoid algebra of $\mathcal{L}$. Then \[
    R_k(\mathcal{L})^{U_k}=\bigcup_{\mu\in\mathcal{L}}k[G_k/K_k]_{\le\mu}^{U_k}\cong k[\mathcal{L}].
    \]
    Here the last isomorphism follows from the closeness of $\CL$. The algebra $k[\mathcal{L}]$ is integrally closed since $\CL$ is saturated (cf. \cite{CLS}*{Theorem~1.3.5}). Since $\CL$ is finitely generated by assumption, this shows $k[\mathcal{L}]$ is finitely generated.
 \end{proof}   
    
    We then define the normal affine $G_k$-variety
\begin{equation*}
V_k(\mathcal{L})=Spec\,R_k(\mathcal{L}).
\end{equation*}
 The rest of this subsection is devoted to prove the following theorem. 
\begin{thm}\label{thm:semigroup}
    The map $\mathcal{L}\mapsto V_k(\mathcal{L})$ is a bijection between the set of closed saturated submonoids of $\iwtp$ with the set of affine embeddings of $G_k/K_k$ (up to isomorphism).
\end{thm}

\begin{rem}
Isomorphisms between embeddings $V_k$ and $V_k'$ of $G_k/K_k$ are $G_k$-equivariant isomorphisms $V_k \cong V_k'$ that restrict to the identity map on $G_k/K_k$. 
\end{rem}

\begin{lemma}\label{lem:emb1}
    Let $\mathcal{L}$ be a closed saturated submonoid of $\iwtp$.  Then $V_k(\mathcal{L})$ is an affine embedding of $G_k/K_k$.
\end{lemma}

\begin{proof}
We have already shown $V_k(\mathcal{L})$ is a normal affine variety. The $G_k$-action is the obvious one. We show $V_k(\CL)$ contains $G_k/K_k$ as an open subspace. 

Let $F_k(\mathcal{L})$ be the field of fractions of $R_k(\mathcal{L})$. It suffices to show $F_k(\mathcal{L})=k(G_k/K_k)$, or $F_k(\mathcal{L})$ contains $k[G_k/K_k]$. Take any $f\in k[G_k/K_k]$. Assume that $f$ belongs to $k[G_k/K_k]_{\le\lambda}$, for some $\lambda\in\iwtp$. Since $\CL$ generates $\breve{X}$, we can write $\lambda=\lambda'-\lambda''$, for some $\lambda'$ and $\lambda''$ in $\mathcal{L}$. Recall the element $\chi_{\lambda''}\in k[G_k/K_k]_{\leq\lambda''}^{U_k}$. Then $\chi_{\lambda''}\in R_k(\mathcal{L})$. Then we have $$\chi_{\lambda''}f\in k[G_k/K_k]_{\leq\lambda'}\subset R_k(\mathcal{L}).$$
    Therefore $f$ belongs to $F_k(\mathcal{L})$. We complete the proof.
\end{proof}

We next construct the inverse of the bijection. Let $V_k$ be an affine embedding of $G_k/K_k$, and let $k[V_k]$ be the ring of regular functions on $V_k$. Then $k[V_k]$ is identified with a $G_k$-subalgebra of $k[G_k/K_k]$. We set 
\begin{equation*}
    \mathcal{L}(V_k)=\{\mu\in X\mid k[V_k]^{(\mu)}\neq 0\}.
\end{equation*}
 It is clear that $\mathcal{L}(V_k)$ is a submonoid of $\iwtp$. We also have $V_k^{U_k} \cong k[\CL(V_k)]$.

\begin{lemma}\label{lem:emb2}
   Let $V_k$ be an affine embedding of $G_k/K_k$. The submonoid $\mathcal{L}(V_k)$ is  a  closed and saturated.
\end{lemma}

\begin{proof}
We assume $V_k \neq G_k/K_k$, otherwise the claim follows by Theorem~\ref{thm:int}. By the multiplicity-free property, we know that $k[V_k]^{U_k}$ is isomorphic to the monoid algebra $k[\mathcal{L}(V_k)]$. Since $k[V_k]$ is finitely generated and normal, we deduce that $k[V_k]^{U_k}\cong k[\mathcal{L}(V_k)]$ is also finitely generated and normal (cf. \cite{Gross}*{Theorem~9} and \cite{Gross}*{proof of Theorem~17}). Hence $\mathcal{L}(V_k)$ is also saturated (cf. \cite{CLS}*{Theorem~1.3.5}).

We next show that $\mathcal{L}(V_k)$ is closed. Take $\lambda\in\iwtp$ and $\mu\in\mathcal{L}(V_k)$, with $\lambda\leq\mu$. Recall the eigenfunctions $\chi_\lambda$ and $\chi_\mu$ in $k[G_k/K_k]$ from Section \ref{sec:BE}. Let $D=V_k-G_k/K_k$ be the closed subvariety. Let $D_1,D_2,\dots,D_s$ be the irreducible components of $D$ with codimension one, and let $v_i$ be the $G_k$-invariant valuation corresponding to $D_i$, for various $i$. Recall $V_k$ is normal. Then by \cite{Hart}*{Proposition II.6.3 A}, for $f\in k[G_k/K_k]$, one has $f\in k[V_k]$ if and only if $v_i(f)\ge 0$, for all $i$. In particular we have $v_i(\chi_\mu)\ge 0$ for all $i$. By Proposition \ref{prop:va}, we have $v_i(\chi_{\overline{\alpha_j}})\leq 0$, for $1\leq i\leq s$ and $j\in \Ir$. Since $\lambda\leq\mu$, by Lemma \ref{le:min} we have
\[
2(\mu-\lambda)=\sum_{j\in\Ir}n_j\overline{\alpha_j}, \text{\qquad where $n_j\in\mathbb{Z}_{\ge 0}$.}
\]
Therefore we have
\[
2v_i(\chi_\lambda)=v_i(\chi_{2\lambda})=v_i(\chi_{2\mu})-\sum_{j\in\Ir}n_jv_i(\chi_{\overline{\alpha_j}})\geq v_i(\chi_{2\mu})\geq 0.
\]
Hence we deduce that $v_i(\chi_\lambda)\ge 0$, and therefore $\chi_\lambda\in k[V_k]$. This implies that $\lambda$ belongs to $\mathcal{L}(V_k)$. We conclude that $\mathcal{L}(V_k)$ is closed.

It remains to show that $\mathcal{L}(V_k)$ generates $\breve{X}$. Thanks to Lemma \ref{le:dec}, it suffices to show that for $\mu\in\iwtp$, there are $\mu'$ and $\mu''$ in $\mathcal{L}(V_k)$, such that $\mu=\mu'-\mu''$. Let $I(D)\subset k[V_k]$ be the defining ideal of the closed subset $D$. Since $I(D) \subset k[V_k]$ is a locally finite $G_k$-module, it must contain non-zero $B_k$-eigenfunctions by Lie-Kolchin theorem. Suppose $\chi_\gamma$ belongs to $I(D)$, where $\gamma\in \mathcal{L}(V_k)$. Then $v_i(\chi_\gamma)>0$, for $1\le i\le s$. Take $n\in\mathbb{N}$, such that $v_i(\chi_\mu)+nv_i(\chi_\gamma)\ge 0$, for all $i$. Then $\chi_\mu\chi_{\gamma}^n=\chi_{\mu+n\gamma}$ is regular on $V_k$. Therefore $\mu=(\mu+n\gamma)-n\gamma$ is a desired decomposition. We complete the proof.
\end{proof}

Finally, we show that the maps $\mathcal{L}\mapsto V_k(\mathcal{L})$ and $V_k\mapsto\mathcal{L}(V_k)$ are mutually inverse to each other. 
\begin{lemma} 
    \begin{enumerate}
    \item Let $\mathcal{L}$ be a closed saturated submonoid of $\iwtp$. Then $\mathcal{L}=\mathcal{L}(V_k(\mathcal{L}))$.
    \item     Let $V_k$ be an affine embedding of $G_k/K_k$. We have $V_k(\mathcal{L}(V_k)) \cong V_k$.
    \end{enumerate}
\end{lemma}

\begin{proof}
We show Part (1). By definition, we have $k[V_k(\mathcal{L})]=\bigcup_{\mu\in\mathcal{L}}k[G_k/K_k]_{\le \mu}$.
Since $\CL$ is closed, the claim follows from Lemma \ref{le:Uk}.

We show Part (2) now. By Lemma \ref{le:Uk} and the definition of $\mathcal{L}(V_k)$, the space $k[V_k]^{U_k}$ is spanned by $\{\chi_\mu\mid \mu\in\mathcal{L}(V_k)\}$. Let $A$ be the smallest $G_k$-invariant subalgebra of $k[G_k/K_k]$, which contains $\chi_\mu$, for all $\mu\in\mathcal{L}(V_k)$. Let $\overline{A}$ be the integral closure of $A$ in $k(G_k/K_k)$. By \cite{Gross}*{Theorem 5}, the algebra $k[V_k]$ is integral over $A$. So we have $$A\subset k[V_k]\subset \overline{A}\subset k(G_k/K_k).$$
In particular $\overline{A}$ is integrally over $k[V_k]$. Since $k[V_k]$ is integrally closed, we deduce that $k[V_k]=\overline{A}$.

On the other hand, by Lemma \ref{le:Uk} and the definition \eqref{eq:R}, we know that the subspace $R_k(\mathcal{L}(V_k))^{U_k}$ is also spanned by $\{\chi_\mu\mid \mu\in \mathcal{L}(V_k)\}$. The algebra $R_k(\mathcal{L}(V_k))$ is integrally closed, thanks to Lemma \ref{lem:emb1}. Therefore by the similar arguments, we have $R_k(\mathcal{L}(V_k))=\overline{A}$. Hence we deduce that $k[V_k]=R_k(\mathcal{L}(V_k))$. This completes the proof.
\end{proof}

We have now proved Theorem~\ref{thm:semigroup}. Let us derive some immediate consequence.

\begin{cor}
Assume $G_k$ is semisimple. Then any affine embedding $V_k$ of $G_k/K_k$ is trivial, that is, $V_k \cong G_k/K_k$.
\end{cor}
\begin{proof}
Let $\CL = \CL(V_k)$. We show $\CL = \breve{X}^+$.  Since $G_k$ is semisimple, the simple roots form a $\BQ$-basis of $X \otimes_\BZ \BQ$. Therefore $\{\overline{\alpha}_i \vert i \in \Ir\}$ forms a $\BQ$-basis of $\breve{X}\otimes_{\BZ} \BQ$. 

We know $\CL$ generates $\breve{X}$. So we can find $\lambda \in \CL$ such that $\langle \alpha_i^{\vee}, \lambda \rangle >0$ for all $i \in \Ir$. We write $\lambda = \sum_{i\in\Ir,c_i \in \BQ} c_i \overline{\alpha}_i$. Let $A = (\langle \alpha_j^\vee, \overline{\alpha}_i \rangle)_{i,j \in \Ir}$.   Then we claim $A$ is a generalized Cartan matrix of finite type in the sense of \cite{Kac}*{Theorem~4.3}. It is clear $A$ is a generalized Cartan matrix. Let $\rho = \sum_{i \in I} a_i \alpha^\vee_i$ be the half sum of positive coroots of $G_k$, and denote $\rho' = \sum_{i \in \Ir} a_i \alpha^\vee_i$. Let $u = (a_i)_{i \in \Ir}$. Then all coordinates of both $u$ and $Au$ are positive. This shows $A$ is of finite type by \cite{Kac}*{Theorem~4.3}.

Now $\langle \alpha_i^{\vee}, \lambda \rangle >0$ for all $i \in \Ir$ implies that $c_i >0$ for all $i$ by \cite{Kac}*{Theorem~4.3}. Then for any $\mu \in \breve{X}^+$, we have $\mu < N \lambda$ for some $N >0$. Since $N \lambda \in \CL$ and $\CL$ is saturated, we have $\mu \in \CL$ as well. This finishes the proof.
\end{proof}
\begin{remark}
Up to a normalization, the matrix $A = (\langle \alpha_j^\vee, \overline{\alpha}_i \rangle)_{i,j \in \Ir}$ is the Cartan matrix of  spherical root system considered in \S\ref{subsec:spherical}. One can indeed use the Cartan matrix of the spherical root system for the proof as well. 
\end{remark}

\subsection{$G_k$-orbits on $V_k(\CL)$} 

Let $\CL$ be a closed saturated submonoid of $\breve{X}^+$ and $V_k(\CL)$ be the associated affine embedding of $G_k/K_k$. Let $\mathfrak{p}$ be a non-zero $G_k$-invariant prime ideals of $k[V_k(\CL)]$. Let $(\emptyset\neq)\CL(\mathfrak{p}) = \{\mu \in \breve{X}^+ \vert \mathfrak{p}^{(\mu)} \neq 0 \}$.
Since $\mathfrak{p}$ is a prime ideal, the set-theoretical difference $\CL \backslash \CL(\mathfrak{p})$ is a subsemigroup of $\CL$. 

\begin{lem}\label{lem:Ukp}
We have a short exact sequence of vector spaces
\[
 0 \rightarrow \mathfrak{p}^{U_k} \rightarrow k[V_k(\CL)]^{U_k} \rightarrow (k[V_k(\CL)]/\mathfrak{p})^{U_k} \rightarrow 0.
\]
\end{lem}

\begin{proof}
It suffices to show the surjectivity part. When char $k = 0$, the claim is trivial by the complete reducibility of $k[G_k]$. We assume char $k = p  >0$. 

Let $f \in (k[V_k(\CL)]/\mathfrak{p})^{U_k}$ be an $B_k$-eigenfunction of weight $\chi_f \in {X}^+$. By \cite{LV}*{Theorem~1.1}, we can find a $B_k$-eigenfunction $F \in k[V_k(\CL)]$ such that $F = f^{p^N}$ in $k[V_k(\CL)]/\mathfrak{p}$ for some $N \in \BZ_{>0}$. 

Since $\CL$ is saturated and $k[V_k(\CL)]$ is multiplicity free, we can find $H \in k[V_k(\CL)]$ such that $H^{p^N} = F$. Therefore $H^{p^N} = f^{p^N}$ in $k[V_k(\CL)]/\mathfrak{p}$. Since char $k = p$ and $\mathfrak{p}$ is prime, we conclude that $H = f$. This finishes the proof. 
\end{proof}

\begin{cor}\begin{enumerate}
\item We have $\CL \backslash \CL(\mathfrak{p}) = \{\mu \in \breve{X}^+ \vert (k[V_k(\CL)]/\mathfrak{p})^{(\mu)} \neq 0\}$.
\item We have $(k[V_k(\CL)]/\mathfrak{p})^{U_k} \cong k[\CL\backslash \CL(\mathfrak{p})]$.
\end{enumerate}
\end{cor}

For any closed subsemigroup $\CJ$ of $\CL$, we define the $k$-subspace $R_k(\CJ)$ of $R_k(\CL)$ by 
\[
R_k(\CJ) = \bigcup_{\mu \in \CJ}k[G_k/K_k]_{\le \mu}.
\]

\begin{lem}\label{lem:RJ}
 If $\CJ$ is a closed prime ideal of $\CL$, then $R_k(\CJ)$ is a prime ideal of $R_k(\CL)$. 
\end{lem}

\begin{proof}
 We first show $R_k(\CJ)$ is an ideal. It is obviously closed under addition. Let $g \in R_k(\CL)$ and $f \in R_k(\CJ)$. Then $g \in k[G_k/K_k]_{\le \lambda}$ for some $\lambda \in \CL$ and $f \in k[G_k/K_k]_{\le \mu}$ for some $\mu \in \CL$.  Then we have $gf \in k[G_k/K_k]_{\le \lambda + \mu}$ by Theorem \ref{thm:int} (3). Since $\CJ$ is an ideal of $\CL$, then $\lambda + \mu \in \CJ$. Now we have $gf \in R_k(\CJ)$, hence $R_k(\CJ)$ is ideal. 

It follows by Theorem~\ref{thm:int} that $R_k(\CJ)$ admits a good filtration as a $G_k$-module. Then by Lemma~\ref{lem:Ukp} and \cite{Gross}*{Theorem~12}, we have that $R_k(\CJ)$ is a prime ideal.
\end{proof}

\begin{lem}\label{lem:Lp}
The subsemigroup $\CL(\mathfrak{p})$ is a closed prime ideal of $\CL$.
\end{lem}

\begin{proof}
It is clear that $\CL(\mathfrak{p})$ is a prime ideal of $\CL$. We show it is closed.

Let $\lambda \in  \CL(\mathfrak{p})$ and $\mu \in \breve{X}^+$ with $ \mu \le \lambda$. By Lemma~\ref{le:min}, we write $\lambda- \mu  = \frac{1}{2} \sum_{i\in\Ir}n_i \overline{\alpha_i}$ for $n_i \ge 0$. 
By \cite{Mat}*{Theorem~10.2} and \cite{LV}*{Lemma~1.4}, there exists a $G_k$-invariant valuation $\nu$ on $k(G_k/K_k)$ such that 
\[
\mathfrak{p} = \{f \in k[V_k(\CL)] \vert \nu(f) > 0\}.
\]
Then we have 
\[
\nu(\chi_{\mu}) = \nu(\chi_{\lambda}) - \frac{1}{2} \sum_{i\in\Ir}n_i \overline{\alpha_i} \nu(\chi_{\overline{\alpha_i}}).
\]
Hence $\nu(\chi_{\mu }) > 0$ by Proposition~\ref{prop:va}. This shows $\mu \in \CL(\mathfrak{p})$.
\end{proof}

\begin{cor}\label{cor:Lp}
We have 
$\mathfrak{p} = R_k(\CL(\mathfrak{p}))$.
\end{cor}

\begin{proof}

Let $A = \mathfrak{p} \cap  R_k(\CL(\mathfrak{p}))$. Then $A$ is a prime ideal by Lemma~\ref{lem:RJ}. Then we consider the short exact sequence of $G_k$-modules 
\[
 0 \rightarrow A \rightarrow \mathfrak{p} \rightarrow \mathfrak{p}/A  \rightarrow 0.
\]
We claim the $U_k$-invariants are also exact, that is, 
\[
 0 \rightarrow A^{U_k} \rightarrow \mathfrak{p}^{U_k} \rightarrow (\mathfrak{p}/A)^{U_k} \rightarrow 0.
\]

 Let us still assume  char $k = p > 0$, otherwise the claim is trivial. We only need to show the surjectivity. Let $f \in (\mathfrak{p}/A)^{U_k} \subset (R_k(\CL)/A)^{U_k}$ be an $B_k$-eigenfunction of weight $\lambda \in X^+$.  Let $H \in R_k(\CL)$ be the $B_k$-eigenfunction such that $H + A = f$ in $ \mathfrak{p}/A$ constructed in Lemma~\ref{lem:Ukp}. Then we have $H \in \mathfrak{p}$. This shows the exactness. 

Note that $\mathfrak{p}^{U_k} = R_{k}(\CL(\mathfrak{p}))^{U_k}$. Hence $A^{U_k} = \mathfrak{p}^{U_k}$, which implies that $(\mathfrak{p}/A)^{U_k} = 0$. Therefore $\mathfrak{p}/A = 0$ as an algebraic $G_k$-module. This shows $\mathfrak{p} \subset R_{k}(\CL(\mathfrak{p}))$. One can symmetrically obtain the other inclusion. We finish the proof now.
\end{proof}

\begin{cor}
Let $Y$ be a $G_k$-orbit closure of $V_{k}(\CL)$. Then $Y$ is normal whose coordinate ring $k[Y]$ admits a good filtration as a $G_k$-module.
\end{cor}

\begin{proof}
We know $k[Y] = R_{k}(\CL)/R_{k}(\CJ)$ for some closed prime ideal $\CJ$ of $\CL$ by Lemma~\ref{lem:Lp} and Corollary \ref{cor:Lp}. Since $R_{k}(\CJ)$ has a good filtration, the quotient $k[Y]$ has a good filtration. Then by \cite{Gross}*{Theorem~17}, the normality of $k[Y]$ is equivalent to the normality of $k[Y]^{U_k} \cong k[\CL\backslash\CJ]$. 

 Since $\CL$ is saturated and $\CJ$ is an prime ideal, the normality of $k[\CL\backslash\CJ]$ is clear by the theory of toric varieties. 
\end{proof}

We summarize the results in this subsection by the following theorem.

\begin{thm}
\begin{enumerate}\label{thm:CJ}
    \item The map $\mathfrak{p} \rightarrow \CL(\mathfrak{p})$ is an order-preserving bijection between non-zero $G_k$-invariant prime ideals, and closed prime ideals in $\CL$. 

    \item We have an order-reversing bijection between (the closure of) non-zero $G_k$-orbits of $V_k(\CL)$ and closed prime  ideals in $\CL$.  
        \item Let $Y$ be a $G_k$-orbit closure of $V_k(\CL)$. Then its coordinate ring $k[Y]$ admits a good filtration as a $G_k$-module.
    \item Let $Y$ be a $G_k$-orbit closure of $V_k(\CL)$. Then $Y$ is normal. 
\end{enumerate}
\end{thm}

\begin{remark}
The normality for $G_k$-orbit closures of $V_k(\CL)$ has been obtain by Tange \cite{Tange} by different methods. 
\end{remark}

\subsection{The integral model and quantization}

We construct the integral model for the affine embedding $V_k(\mathcal{L})$, as well as the integral model for the closures of $G_k$-orbits in $V_k(\mathcal{L})$.

Recall the non-commutative $\mathbb{Z}[q,q^{-1}]$-algebra $\mathbf{O}_q(G/K)$, the commutative ring $\mathbf{O}(G/K)$ and their filtration in Section \ref{sec:int}. 

\begin{definition}\label{def:caq}
    Let $\mathcal{L}$ be a closed saturated submonoid of $\iwtp$ that generates $\breve{X}$. We define
\begin{equation*}
    \mathbf{R}(\mathcal{L})=\bigcup_{\mu\in\mathcal{L}}\mathbf{O}(G/K)_{\leq\mu}, \quad \text{ and }\quad \mathbf{R}_q(\mathcal{L})=\bigcup_{\mu\in\mathcal{L}}\mathbf{O}_q(G/K)_{\leq\mu}.
\end{equation*}
It is clear $\mathbf{R}(\mathcal{L})$ (resp., $\mathbf{R}_q(\mathcal{L}))$ is spanned by a subset of the dual canonical basis of $\mathbf{O}(G/K)$ (resp., $\mathbf{O}_q(G/K)$) as a free $\mathbb{Z}$ (resp., $\mathbb{Z}[q,q^{-1}]$)-module.
\end{definition}

 Let $k$ be an algebraically closed field with characteristic $\neq 2$ as before. Thanks to Theorem \ref{thm:int}, we have the canonical isomorphism of $k$-algebras
\begin{equation}\label{eq:bas}
k\otimes_\mathbb{Z}{\mathbf{R}(\mathcal{L})}\cong\bigcup_{\mu\in\mathcal{L}}k[G_k/K_k]_{\leq\mu}=R_k(\mathcal{L}).
\end{equation}

Define the affine scheme 
\begin{equation*}
    \mathbf{V}(\mathcal{L})=Spec\,\mathbf{R}(\mathcal{L}).
\end{equation*}
For any closed  subsemigroup $\CJ \subset \CL$, we define $\mathbf{R}(\CJ) = \bigcup_{\mu \in \CJ} \mathbf{O}(G/K)_{\le \mu}$. It follows from Lemma~\ref{lem:RJ} (by taking $k =\mathbb{C}$) that $\mathbf{R}(\CJ)$ is a prime ideal of $\mathbf{R}(\CL)$. We write $\mathbf{R}(\CL\backslash\CJ) = \mathbf{R}(\CL)/\mathbf{R}(\CJ)$.  We define the affine scheme  
\[
V(\CL\backslash \CJ) = Spec\, \mathbf{R}(\CL\backslash\CJ).
\]

\begin{thm}\label{thm:VL}Let $k$ be any algebraically closed field such that char $k \neq 2$.
    \begin{enumerate}  
        \item  The geometric fiber of $\mathbf{V}(\CL)$ at $k$ is precisely the affine embedding $V_k(\CL)$ of $G_k/K_k$ associated with the semigroup $\CL$.
        \item The geometric fiber of $\mathbf{V}(\CL\backslash \CJ)$ at $k$ is precisely the closure of the $G_k$-orbit on $V_k(\CL)$ of associated with the submonoid $\CJ$. In particular, the closure of any $G_k$-orbit on $V_k(\CL)$ is defined over $\mathbb{Z}$.
    \end{enumerate}
\end{thm}

\begin{proof}
    Part (1) follows from Theorem~\ref{thm:semigroup} and \eqref{eq:bas}. We show Part (2). It follows by Theorem~\ref{thm:int} (see more details in \cite{BS}) the embedding $\mathbf{R}(\CJ) \rightarrow \mathbf{R}(\CL)$ is based. In particular, we have the following short exact sequence after base change 
    \[
    0 \rightarrow \mathbf{R}(\CJ) \otimes_\BZ k \rightarrow \mathbf{R}(\CL) \otimes_\BZ k \rightarrow \mathbf{R}(\CL\backslash\CJ) \otimes_\BZ k \rightarrow 0. 
    \]

    Then by Theorem~\ref{thm:int}, the short exact sequence becomes 
    \[
        0 \rightarrow R_k(\CJ)  \rightarrow R_k(\CL)  \rightarrow R_k(\CL\backslash\CJ)  \rightarrow 0.
    \]
    Now Part (2) follows from Theorem~\ref{thm:CJ}. 
\end{proof}

     It follows from Theorem~\ref{thm:int} that $\mathrm{B}(\mathcal{L)}$ specializes to a basis of $R_k(\CL)$, which we still denote by $\mathrm{B}(\mathcal{L)}$.
\begin{defi}\label{defi:dualCBL}
We call $\mathrm{B}(\mathcal{L)}$ the dual canonical basis of the embedding $V_k(\CL)$.
\end{defi}

\begin{remark}\label{remark:groupcase}
We can consider $G_k$ as a symmetric space $(G_k \times G_k )/ \Delta(G_k)$. Here $\Delta(G_k)$ denotes the diagonal embedding of $G_k$. Affine embeddings of $G_k$ are precisely reductive monoids over $G_k$ by \cite{Rit}*{Proposition~1\&Lemma~2}. So all results in this section apply. In particular, we recover various results in \cite{Vin} (for characteristic $0$) and in \cite{Rit} (for arbitrary characteristic). 
\end{remark}

\subsection{Abelianzation}\label{sec:ab}

Let $G_{k,0}$ be the commutator subgroup of $G_k$, and $Z_k$ be the connected center of $G_k$. The group involution $\theta_k$ leaves $G_{k,0}$ and $Z_k$ invariant. For an affine embedding $V_k$ of $G_k/K_k$, its \emph{abelization} $A_k=A_k(V_k)$ is the (affine) GIT quotient
\[
A_k=V_k/\!\!/G_{k,0}.
\]
By definition, $A_k$ is the spectrum of the subalgebra $k[V_k]^{G_{k,0}}$ of $k[V_k]$ 
consisting of $G_{k,0}$-invariant functions in $k[V_k]$. There is a canonical map $\pi:V_k\rightarrow A_k$.  

\begin{remark}
When $V_k$ is an affine embedding of $G_k$, then $V_k$ is a reductive monoid. In this case, the quotient $A_k$ is a commutative monoid. This is the origin of the name ``abelianization".
\end{remark}

Let $T_{k,0}=T_k\cap G_{k,0}$ be the maximal torus of $G_{k,0}$. Let $X_0$ and $X_Z$ be the groups of characters of $T_{k,0}$ and $Z_k$, respectively. Then we have an embedding $X\hookrightarrow X_0\times X_Z$ by restricting characters of $T_k$ to $T_{k,0}$ and $Z_k$. We will identify $X$, hence $\iwt \subset X$, with its image in $X_0\times X_Z$. 

Let $\mathcal{L}=\mathcal{L}(V_k)$ be the submonoid associated with $V_k$. Let $\mathcal{L}_Z\subset X_Z$ be the submonoid such that $(0,\mu)\in\mathcal{L}$ if and only if $\mu\in\mathcal{L}_Z$. We define a partial order on $\CL$ by saying $M_1 \le_Z M_2$ if $M_2 - M_1 \in \CL_Z$. Let $\CM$ be the set of minimal elements in $\CL$ with respect to this new partial order. We write $M_1 \sim M_2$ if $M_1 \le_Z M_2$ and $M_2 \le_Z M_1$. 
 Let $\CM_0 = \CL_Z \cap (-\CL_Z)$ be the greatest subgroup contained in $\CL_Z$ (and in $\CL$). Then we have $\CL = \CM + \CL_Z$.

Then it is clear that the coordinate ring $k[A_k]$ is isomorphic to the monoid algebra $k[\mathcal{L}_Z]$. Since $\mathcal{L}_Z$ is a saturated monoid by definition, $A_k$ is a normal toric variety. 

 \begin{defi}
    The affine embedding $V_k$ is called  \emph{very flat}  if the map $\pi$ is flat, and its scheme-theoretic fibres are reduced and irreducible.
\end{defi}

\begin{remark}
Our very flat embedding is simply referred as flat by Vinberg \cite{Vin}*{Definition~2}. We choose to call such an embedding very flat following Rittatore \cite{Rit}. There are examples where $\pi$ is flat with non-reduced fibers; see \cite{Vin}*{Example~4.2}.
\end{remark}

\begin{prop}\label{prop:flat}
The following are equivalent. 
\begin{enumerate}
    \item $k[V_k]$ is a flat $k[A_k]$-module;
    \item $k[V_k]$ is a free $k[A_k]$-module;
    \item if $M_1 + \chi_1 = M_2 + \chi_2$ ($M_1, M_2 \in \CM$, $\chi_1, \chi_2 \in \CL_Z$), then $M_1 \sim M_2$ (and $\chi_1 \sim \chi_2$);
    \item $k[V_k]$ decomposes as a $k[A_k]$-module into the tensor product
    \[
    k[V_k] = k[A_k]\otimes_k k[G_k/K_k]_{\CM_1},
    \]
    where $\CM_1$ is a set of representative of the cosets of $\CM_0$ in $\CM$,  and 
    $k[G_k/K_k]_{\CM_1}$ is the subspace of $k[V_k]$ spanned by the dual canonical basis elements $B(G/K)_\mu$, for $\mu\in \CM_1$. 
\end{enumerate}
\end{prop}
\begin{proof}
The formulation of the proposition is essentially the same as \cite{Vin}*{Proposition~3}. Note that $k[V_k]^{U_k} = k[\CL(V_k)]$ by Theorem~\ref{thm:semigroup}. Also note that we are only interested in the decomposition of $k[V_k]$ as $k[A_k]$-modules, while ignoring the $G_k$-action. 

Now the argument in the proof of \cite{Vin}*{Proposition~3} applies as well. 
\end{proof}

\begin{prop}\label{prop:fiber}
Assume $\pi$ is flat. Then the fibers of $\pi$ are reduced and irreducible if and only if $\CM$ is a submonoid of $\CL$.
\end{prop}

\begin{proof}
We write 
\begin{equation}\label{eq:flat}
 k[V_k] = \bigoplus_{\mu \in X_z}  k[V_k]_{\mu}, \quad \text{ where }k[V_k]_{\mu} = \bigcup_{(\lambda, \mu) \in \mathcal{L}}k[V_k]_{\le (\lambda,\mu)}.
\end{equation}
Then $k[V_k]_{\mu}$ is $G_{k,0}$-stable and admits a good filtration as a $G_{k,0}$-module. 

Let $\{\gamma_1,\gamma_2,\cdots,\gamma_k\}\subset \iwt_0^+$ be the set of minimal dominant weights (with respect to the usual partial order $\le$). For $1\le i\le k$, let $\Gamma_i\subset X_Z$ be the subset such that $\mu\in \Gamma_i$ if and only if $(\gamma_i,\mu)\in\mathcal{L}(V_k)$. Note that $\Gamma_i+\mathcal{M}_0=\Gamma_i$.  We inductively construct a set $\Gamma_i'\subset\Gamma_i$ of representatives of cosets $\mathcal{M}_0$ in $\Gamma_i$, such that if $r_i+\mathcal{M}_0=r_j+\mathcal{M}_0$ for some $r_i\in\Gamma_i'$ and $r_j\in\Gamma_j'$, then $r_i=r_j$, for any $1\le j \le i\le k$.

We then define
$\mathcal{M}_1=\{(\lambda,\mu)\in \mathcal{M}\mid \lambda\geq\gamma_i,\mu\in\Gamma_i',\text{ for some }1\leq i\leq k\}$.
It is then direct to see that $\mathcal{M}_1$ is a set of representatives of cosets $\mathcal{M}_0$ in $\mathcal{M}$. Moreover, if $\lambda, \mu \in \CM$ with $\mu \le \lambda$ and $\lambda \in \CM_1$, then $\mu \in \CM_1$. By Proposition~\ref{prop:flat} again, we have  $k[V_k] = k[A_k]\otimes_k k[G_k/K_k]_{\CM_1}$.

Let $x \in A_k$ with the defining maximal ideal $\mathfrak{m} \subset k[A_k]$. Let $\mathfrak{p}$ be the ideal in $k[V_k]$ generated by $\mathfrak{m}$. By construction, the ideal $\mathfrak{p}$ is $G_{k,0}$-invariant. The set of points in $A_k$ where $\mathfrak{p}$ is prime is $T_k$-stable. Furthermore, it is open by \cite{Stack}*{Section~37.26}.

Let $e_0\in A_k$ be the point defined by 
\[
\chi(e_0)=\left\{\begin{array}{ll}
    1, & \text{if $\chi\in \CM_0$;} \\
    0, & \text{if $\chi\in \mathcal{L}_Z-\CM_0$.}
\end{array}\right.
\]
It suffices to study the fiber $\pi^{-1}(e_0)$, since it is in the closed $T_k$-orbit of $A_k$.   Let $\mathfrak{m}\subset k[A_k]$ be the defining ideal of $e_0$, and $\mathfrak{p}\subset k[V_k]$ be the ideal generated by $\mathfrak{m}$.

{\it (a) We claim the following sequence is exact  
\[
 0 \rightarrow \mathfrak{p}^{U_k} \rightarrow k[V_k(\CL)]^{U_k} \rightarrow (k[V_k(\CL)]/\mathfrak{p})^{U_k} \rightarrow 0.
\]}

It suffices to show the surjectivity of the quotient map. 

We write $k[V_k(\CL)]= \mathfrak{p}_1 \oplus \mathfrak{p}_2 \oplus \mathfrak{p}_3$ as vector spaces. Here $\mathfrak{p}_1$ is spanned by $B(G/K)_\mu$ for $\mu\in \mathcal{L}-\CM$, $\mathfrak{p}_2$ is spanned by $(e^\chi-1)B(G/K)_\gamma$ for $\gamma\in \CM$ and $\chi\in \CM_0$, and $\mathfrak{p}_3$ is spanned by $B(G/K)_\zeta$ for $\zeta\in \CM_1$.  Then $\mathfrak{p} = \mathfrak{p}_1+ \mathfrak{p}_2$ following Proposition~\ref{prop:flat} (4).

Let $\zeta \in \CM_1$ and $\mu \in \CL$ with $\mu \le \zeta$. Then either $\mu \in \CL- \CM$ or $\mu \in \CM_1$ by our choice of $\CM_1$. It follows that $\mathfrak{p}_1 + \mathfrak{p}_3$ is stable under the $G_{k,0}$-action, since the span of $\cup_{\mu \le \zeta} B(G/K)_\mu$ is stable. Moreover, the space $\mathfrak{p}_1 + \mathfrak{p}_3$ admits a good filtration. Since $\mathfrak{p}_1$ admits a good filtration as a $G_{k,0}$-module as well, we have the following short exact sequence by \cite{Gross}*{Section~5}: 
\[
 0 \rightarrow \mathfrak{p}_1^{U_k} \rightarrow (\mathfrak{p}_1+\mathfrak{p}_3)^{U_k} \rightarrow ((\mathfrak{p}_1+\mathfrak{p}_3)/\mathfrak{p}_1)^{U_k} \rightarrow 0.
\]

On the other hand, we have $k[V_k(\CL)]/\mathfrak{p} \cong (\mathfrak{p}_1 + \mathfrak{p}_3) / \mathfrak{p}_1 \cong \mathfrak{p}_3$  as $G_{k,0}$-modules. This proves claim (a).

Assume $\CM$ is not a submonoid of $\CL$. Then we can find $M_1, M_2 \in \CM$ such that $M_1 + M_2 \in \CL- \CM$. Then $k[V_k(\CL)]/\mathfrak{p}$ can not be an integral domain. 

Assume $\CM$ is a submonoid of $\CL$. Then $\CM_0 \subset \CM$ is a subgroup. We denote by $\CM/\CM_0$ the quotient monoid. Then $(k[V_k(\CL)]/\mathfrak{p})^{U_k} \cong k[\CM/\CM_0]$ by claim (a). This is an integral domain. Then by \cite{Gross}*{Theorem~12}, the quotient $k[V_k(\CL)]/\mathfrak{p}$ is also an integral domain. We conclude that $\mathfrak{p}$ is prime.  
\end{proof}

%%%%%%%%%%%%%%%%%%%%%%%%%%%%%%%%%%%%%%%%
%%%%%%%%%%%%%%%%%%%%%%%%%%%%%%%%%%%%%%%%

%%%%%%%%%%%%%%%%%%%%%%%%%%%%%%%%%%%%%%%%
%%%%%%%%%%%%%%%%%%%%%%%%%%%%%%%%%%%%%%%%
\section{The canonical embedding}\label{sec:can}

Throughout this section, we assume that $G_k$ is  semisimple. 

\subsection{The enveloping variety}

Recall that $T_k$ is a $\theta_k$-stable maximal torus containing a $\theta_k$-split maximal torus. We consider 
$$
\widetilde{G}_k=G_k\times T_k\qquad\text{and}\qquad \widetilde{\theta}_k=\theta_k\times(\theta_k)_{\mid_{T_k}}.
$$ 
Then $T_k\times T_k$ is a maximal torus of $\widetilde{G}_k$ containing a maximal $\widetilde{\theta}_k$-split torus, and the corresponding weight lattice is $X\times X$. We have $(\mu,\lambda)\le (\mu',\lambda')$ in $X\times X$ if and only if  $\lambda = \lambda'$ and  $\mu\le\mu'$ in $X$. Let $\widetilde{K}_k$ be the $\widetilde{\theta}_k$-fixed-point subgroup of $\widetilde{G}_k$. The spherical weight lattice of $\widetilde{G}_k/\widetilde{K}_k$ is $\iwt\times\iwt$. The monoid of the dominant spherical weights is $\iwtp\times \iwt$. Let $\widetilde{B}_k=B_k\times T_k$ be a Borel subgroup of $\widetilde{G}_k$, and let $\widetilde{U}_k=U_k\times\{e\}$ be the unipotent radical of $\widetilde{B}_k$.

Let us define a new partial order $\preceq$ on $X$ by setting $\mu\preceq\lambda$ if and only if $n\mu\le n\lambda$, for some positive integer $n$. For $\mu\in\iwt$, we define 
\begin{equation}\label{eq:kp}
k[G_k/K_k]_{\preceq\mu}=\bigcup_{\lambda\in\iwtp,\;\lambda\preceq\mu}k[G_k/K_k]_{\le\lambda}.
\end{equation}

\begin{remark}
    Suppose $G_k$ is of adjoint type. Then the weight lattice is spanned by the set of simple roots $\{\alpha_i\mid i\in\I\}$. Then the partial order $\preceq$ coincides with the ordinary partial order $\le$ on $X$. Various combinatorics are simplified in this setting.
\end{remark}

Recall $\overline{T}_k=T_k/T_k^{\theta_k}$ and the character group of the quotient torus $\overline{T}_k$ is isomorphic to $\iwt$. We have the canonical isomorphisms 
\begin{equation}\label{eq:ide}
    k[\widetilde{G}_k/\widetilde{K}_k]\cong k[G_k/K_k]\otimes k[T_k/T_k^{\theta_k}]\cong k[G_k/K_k]\otimes k[\iwt].
\end{equation}
Here $k[\iwt]$ is the group algebra of $\iwt$ with the canonical $k$-linear basis $\{e^\mu\mid\mu\in\iwt\}$.
We define $\widetilde{\mathcal{L}}=\{(\mu,\lambda)\in\iwtp\times\iwt\mid \mu\preceq\lambda\}$.

\begin{lemma}
    The subset $\widetilde{\mathcal{L}}$ is a finitely generated closed saturated submonoid of $\iwtp\times\iwt$ generating $\iwt\times\iwt$ 
\end{lemma}

\begin{proof}
It is clear that $\widetilde{\mathcal{L}}$ is a closed saturated submonoid. 

We show $\widetilde{\mathcal{L}}$ is finitely generated.  Thanks to Lemma \ref{le:min}, the map $(\mu,\lambda)\mapsto (\mu,\lambda-\mu)$ defines an isomorphism $\widetilde{\mathcal{L}}\overset{\sim}{\rightarrow}\iwtp\times L$ as monoids, where $L=\iwt\cap\sum_{i\in\Ir}\mathbb{Q}_{\ge 0}\overline{\alpha_i}$ is a submonoid of $\iwt$. Since $\iwt$ is a discrete subgroup of the vector space $\sum_{i\in\Ir}\mathbb{Q}\overline{\alpha_i}$, the monoid $L$ is generated by finitely many elements in $\iwt\cap\{\sum_{i\in\Ir}n_i\overline{\alpha_i}\mid 0\le n_i\le 1\}$ by Gordan's lemma. The monoid $\iwtp$ is finitely generated by Lemma~\ref{lem:emb2}, since $\breve{X}^+ = \mathcal{L}(G_k/K_k)$ for the trivial embedding of $G_k/K_k$. Therefore $\widetilde{\mathcal{L}}$ is finitely generated.

We next show $\widetilde{\mathcal{L}}$ generates $\iwt\times\iwt$. Since $G_k$ is semisimple, the vector space $\mathbb{Q}\otimes_{\mathbb{Z}}\iwt$ is spanned by $\{\overline{\alpha_i}\mid i\in\Ir\}$. For any $(\mu,\lambda)\in \iwtp\times \iwt$, we can take $\gamma\in\sum_{i\in\Ir}\mathbb{N}\overline{\alpha_i}$, such that $\lambda+\gamma-\mu\in\sum_{i\in\Ir}\mathbb{Q}_{\ge 0}\overline{\alpha_i}$, which implies $\mu\preceq\lambda+\gamma$. Therefore we have $(\mu,\lambda)=(\mu,\lambda+\gamma)-(0,\gamma)$, where $(\mu,\lambda+\gamma)$ and $(0,\gamma)$ both belong to $\widetilde{\mathcal{L}}$. Hence $\widetilde{\mathcal{L}}$ generates $\iwt\times\iwt$.

 We complete the proof.
\end{proof}

Set $\widetilde{V}_k=V_k(\widetilde{\mathcal{L}})$ to be the affine embedding of $\widetilde{G_k}/\widetilde{K_k}$ associated with $\widetilde{\mathcal{L}}$.  Under the isomorphism \eqref{eq:ide}, we have
$k[\widetilde{V_k}]=\bigoplus_{\mu\in\iwt}k[G_k/K_k]_{\preceq\mu}e^\mu$. Following Vinberg \cite{Vin} in the group case, we call $\widetilde{V}_k$ the \emph{enveloping variety} associated with the symmetric space $G_k/K_k$.

\begin{remark}
There is slight abuse of terminology. In the group case (see Remark~\ref{remark:groupcase}), the variety $\widetilde{V}_k$ is not the (universal) Vinberg monoid considered in \cite{Vin}*{Theorem 5} in general. Our construction has much simplified combinatorics, and is easier to work with.
\end{remark}

Recall $G_k$ is semisimple by assumption throughout this section. Let us write $Z_k = \{e\} \times T_k$ for the connected center in $\widetilde{G_k}$. It follows that $G_k$ is the commutator subgroup of $\widetilde{G_k}$. We consider the abelination $\pi: \widetilde{V}_k \rightarrow \widetilde{A}_k$, where $\widetilde{A}_k = \widetilde{V}_k/\!\!/G_{k}$. This is a toric variety with the associated torus $\overline{T}_k$. 

\begin{prop}\label{prop:flatV}
The embedding $\widetilde{V}_k$ is very flat.
\end{prop}
\begin{proof}
We follow the notations in Section~\ref{sec:ab}. We have $\widetilde{\CL}_Z= \{\lambda \in \iwt \vert 0 \prec \lambda\}$. We further have $\CM_0 = \widetilde{\CL}_Z \cap (-\widetilde{\CL}_Z) = \{0\}$. Recall $\CM$ denotes the set of minimal elements with respect to the partial order $\le_Z$. 

{\it (a) We claim $\CM = \{(\lambda, \lambda) \vert \lambda \in \iwt^+\}$.}

 It is clear $\CM \supset \{(\lambda, \lambda) \vert \lambda \in \iwt^+\}$. For any $(\mu, \lambda ) \in \CM$, we have $\mu \preceq \lambda$, or equivalently, $ 0 \preceq \lambda -\mu$. Since $\iwt$ is a lattice, we have $\lambda -\mu \in \iwt$. Therefore $(0, \lambda -\mu) \in \widetilde{\CL}_Z$. Since $(\mu, \lambda )$ is minimal, we must have $\lambda -\mu \in \CM_0 = \{0\}$. This proves the claim.

Now the proposition follows from Proposition~\ref{prop:flat} (3) and Proposition~\ref{prop:fiber}.
\end{proof}

\begin{prop}
Let $a \in \widetilde{A}_k$. Then the fibre $\pi^{-1}(a)$ is a normal spherical $G_{k}$-variety.
\end{prop}

\begin{proof}
Note that the actions of $G_k$ and $Z_k$ commute on $\widetilde{V}_k$. Hence $\pi^{-1}(a) \cong \pi^{-1}(ta)$ for any $t \in Z_k$. Let $J \subset \Ir$. By the orbit theory of toric varieties, we can hence assume $a \in \widetilde{A}_k$ is defined by 
\[
 \chi(a) = \begin{cases}
    1, &\text{if } \chi \in \sum_{i \in J}\BQ_{\ge 0}\overline{\alpha}_i  ;\\
    0, &\text{otherwise}.
 \end{cases}
\]

We consider the localization $k[\widetilde{A}_k][e^{-\overline{\alpha}_i} ; i \in J]$, which defines an open affine subvariety $\widetilde{A}^J_k \hookrightarrow \widetilde{A}_k$. The fibre $\pi^{-1}(\widetilde{A}^J_k)$ is an open subvariety $\widetilde{V}_k^J$ of $\widetilde{V}_k$. Note that $\pi: \widetilde{V}_k^J \rightarrow  \widetilde{A}^J_k$ is the abelianization map for the affine embedding $\widetilde{V}_k^J$. Note that $a \in \widetilde{A}^J_k$ is in the closed $Z_k$-orbit of $\widetilde{A}^J_k$. The argument of Proposition~\ref{prop:fiber} applies in the current setting now. 

By Proposition~\ref{prop:flatV}, the scheme-theoretical fiber  $\pi^{-1}(a)$ is reduced and irreducible. The coordinate ring $k[\pi^{-1}(a)]$ admits a good filtration as a $G_k$-module following the proof of Proposition~\ref{prop:fiber}. By Proposition~\ref{prop:fiber} again, we always have a ring isomorphism
\[
 k[\pi^{-1}(a)]^{U_k} \cong k[\CM].
\]

Therefore $k[\pi^{-1}(a)]$ is multiplicity-free and normal, hence a (normal) spherical $G_k$-variety \cite{Gan}*{Theorem~2.8}. 
\end{proof}

Note that $\widetilde{A}_k$ is a monoid with $0$. 
\begin{defi}
The fiber of $0$ under the map $\pi: \widetilde{V}_k \rightarrow \widetilde{A}_k$  is called the asymptotic symmetric space of $G_k/K_k$, denoted by $As (G_k/K_k)$. 
\end{defi}

It follows from the construction that $As(G_k/K_k)$ is a flat deformation of $G_k/K_k$ whose coordinate ring is given by 
\[
 k[As(G_k/K_k)] = {\rm gr}\, k[G_k/K_k].
\]
Here the associated graded ring ${\rm gr}\, k[G_k/K_k]$ is defined similar to \cite{Vinas}.

\subsection{Orbits}

Let $\omega_i$ $(i \in I)$ be the fundamental weights in $X_{\BQ}$. 
\begin{defi}\label{def:EF}
For any $J_1, J_2 \subset \Ir$, we define $\widetilde{\CL}_{J_1, J_2} \subset \widetilde{\CL}$ by 
\begin{align*}
\widetilde{\CL} \backslash \widetilde{\CL}_{J_1, J_2}&= \{(\mu, \gamma) \vert \mu \in \sum_{j \in J_1}\BQ_{\ge0}\overline{\omega}_{j}, \gamma -\mu \in \sum_{j\in J_2}\BQ_{\ge 0}\overline{\alpha}_j\}\\
&= \{(\mu, \gamma) \vert \sum_{j \in J_1} \BQ_{\ge 0}(\overline{\omega}_{j}, \overline{\omega}_{j}) + \sum_{j \in J_2}\BQ_{\ge 0}(0, \overline{\alpha}_j)\}.
\end{align*}

 We say $\widetilde{\CL}_{J_1, J_2}$ is essential if no connected component of the complement of $J_1$ is entirely contained in $J_2$ in terms of the spherical root system in \ref{subsec:spherical}. This generalizes \cite{Vin}*{Definition 4}.
\end{defi}

\begin{rem}\label{rem:EF}
Let us unravel the essential condition here. Let $\widetilde{\CL}_{J_1, J_2}$ be essential. Assume $J_1 \neq \Ir$, otherwise the condition is vacuous.

Let $i$ be in the complement of $J_1$. Then we can write 
\[
 \overline{\omega}_i = \sum_{j \in \Ir} c_j \overline{\alpha}_j, \quad \text{for } c_j \in \BQ_{\ge 0}.
\]
Here $c_j > 0$ if and only if $j$ is in the connected component of $i$. Since the connected component of $i$ is not entirely contained in $J_2$, so there are some $c_j >0 $ with $j \not \in J_2$.
\end{rem}

\begin{prop}\label{prop:Lptilde}
$\widetilde{\CL}_{J_1, J_2} = \widetilde{\CL}(\mathfrak{p})$ for some $\widetilde{G}_k$-stable prime ideal $\mathfrak{p}$  if and only if $\widetilde{\CL}_{J_1, J_2}$ is essential. 

Moreover, all $\widetilde{G}_k$-stable prime ideals of $k[\widetilde{V}_k]$ is of this form. 
\end{prop}

\begin{proof}
Let $\widetilde{\CL}_{J_1, J_2}$ be essential. We show $\widetilde{\CL}_{J_1, J_2} = \widetilde{\CL}(\mathfrak{p})$ for some $\widetilde{G}_k$-stable prime ideal $\mathfrak{p}$  Thanks to Theorem~\ref{thm:CJ}, we show $\widetilde{\CL}_{J_1, J_2}$ is a closed saturated ideal of  $\widetilde{\CL}$. 

(a) We first show  $\widetilde{\CL}_{J_1, J_2}$ is a saturated ideal of $\widetilde{\CL}$. Consider the embedding of monoids 
\[
\psi: \breve{X}^+ \times \breve{X} \rightarrow  \breve{X}^+ \times \breve{X}, \qquad (\mu, \gamma) \rightarrow (\mu, \gamma- \mu).
\]
We see the $\widetilde{\CL} \backslash \widetilde{\CL}_{J_1, J_2}$ is a subsemigroup of $\widetilde{\CL}$. It follows that $\widetilde{\CL}_{J_1, J_2}$ is a  prime  ideal of $\widetilde{\CL}$.

(b) We next show $\widetilde{\CL}_{J_1, J_2}$ is closed. Let $(\mu_1, \gamma_1)$  and $(\mu_2, \gamma_2)$ be in $\widetilde{\CL}$ such that $(\mu_2, \gamma_2) \le (\mu_1, \gamma_1)$ and $(\mu_1, \gamma_1) \in \widetilde{\CL}_{J_1, J_2}$. By definition, we have 
\[
\gamma_1 = \gamma_2 \quad \text{ and } \quad \mu_1 - \mu_2 \in \sum_{i \in \Ir} \BQ_{\ge 0}\overline{\alpha}_i.
\]

Assume $(\gamma_1 -\mu_1) \not \in \sum_{j\in J_2}\BQ_{\ge 0}\overline{\alpha}_j$ first. Then we have $\gamma_2 -\mu_2 = \gamma_1 -\mu_1 + \mu_1 - \mu_2 \not \in \sum_{j\in J_2}\BQ_{\ge 0}\overline{\alpha}_j$ either. So in this case, we have $(\mu_2, \gamma_2) \in \widetilde{\CL}_{J_1, J_2}$.

Assume $\mu_1 \not \in \sum_{j \in J_1}\BQ_{\ge0}\overline{\omega}_{j}$ now. We can further assume $(\gamma_1 -\mu_1) \in \sum_{j\in J_2}\BQ_{\ge 0}\overline{\alpha}_j$, as well as $\mu_2  \in \sum_{j \in J_1}\BQ_{\ge0}\overline{\omega}_{j}$, otherwise we would be done. We write $\mu_1 = \sum_{j \in J_1}d_j \overline{\omega}_j + \sum_{j \not \in J_1}c_j \overline{\omega}_j = \mu_1' + \mu_1''$  for $d_j, c_j \in \BQ_{\ge 0}$.

Then we have 
\begin{align*}
 \gamma_2 - \mu_2 &= (\gamma_1 - \mu_1) + (\mu_1 -\mu_2) \\
 &= (\gamma_1 - \mu_1) + (\mu_1- \mu_1') +  (\mu'_1 -\mu_2) \\
 & \in \sum_{j\in J_2}\BQ_{\ge 0}\overline{\alpha}_j + \mu_1'' + \sum_{j \in \Ir}\BQ_{\ge 0}\overline{\alpha}_j
\end{align*}
Since $\mu''_1 = \sum_{j \in \Ir} c_j \overline{\alpha}_j$ with $c_j >0 $ for some $j \not \in J_2$ by Remark~\ref{rem:EF}, we have 
$ \gamma_2 - \mu_2 \not \in \sum_{j\in J_2}\BQ_{\ge 0}\overline{\alpha}_j$. This shows $(\mu_2, \gamma_2) \in \widetilde{\CL}_{J_1, J_2}$.

Now let $\widetilde{\CJ}$ be a closed prime ideal of $\widetilde{\CL}$. We show $\widetilde{\CJ} = \widetilde{\CL}_{J_1, J_2}$ for some essential $\widetilde{\CL}_{J_1, J_2}$. 

(c) We first show $\widetilde{\CJ} = \widetilde{\CL}_{J_1, J_2}$ for some $J_1, J_2 \subset \Ir$. 

We first see that $\widetilde{\CL}\backslash\widetilde{\CJ}$ is a subsemigroup of $\breve{X}^+ \times \breve{X}$.
Since $\widetilde{\CJ}$ is prime, we see that if $(\mu,\gamma) \in \widetilde{\CL}\backslash\widetilde{\CJ}$ then $\BQ_{\ge 0} (\mu,\gamma) \cap (\breve{X}^+ \times \breve{X}) \subset \widetilde{\CL}\backslash\widetilde{\CJ}$. Let 
\[
J_1 = \{i \in \Ir \vert \BQ_{\ge 0}(\overline{\omega}_i, \overline{\omega}_i) \cap (\widetilde{\CL}\backslash\widetilde{\CJ}) \neq \emptyset\}, J_2 = \{i \in \Ir \vert \BQ_{\ge 0}(0, \overline{\alpha}_i) \cap (\widetilde{\CL}\backslash\widetilde{\CJ}) \neq \emptyset \}.
\]
 It follows that $\widetilde{\CJ} \subset \widetilde{\CL}_{J_1, J_2}$ following Definition~\ref{def:EF}. Since $\widetilde{J}$ is an ideal, it follows from the definition of $J_1$ and $J_2$, we have $\widetilde{\CJ} \supset \widetilde{\CL}_{J_1, J_2}$. We conclude that $\widetilde{\CJ} =\widetilde{\CL}_{J_1, J_2}$.

(d) We next show $\widetilde{\CL}_{J_1, J_2}$ must be essential. 

We have $J_1 \neq \Ir$, otherwise there is nothing to prove.  Assume the contrary, there is a connected component $J_1'$ of the complement of $J_1$ that is completely contained in $J_2$.  So there exists some $\gamma' \in \sum_{j \in J_2}\BQ_{\ge 0} \overline{\alpha}_j \cap \breve{X}$ such that $\langle \alpha^\vee_s,\gamma'\rangle > 0$ for all $s \in J'_1$.

 Let $(\mu, \gamma) \in\widetilde{\CL}_{J_1, J_2}$ be such that $\mu = c_i \overline{\omega}_i $ for some $c_i \in \BZ_{> 0}$ with $ i \in J_1'$, and $\gamma \in \sum_{j \in J_2}\BQ_{\ge 0} \overline{\alpha}_j$. The existence of such  $(\mu, \gamma) \in\widetilde{\CL}_{J_1, J_2}$ is guaranteed by the fact that $J'_1 \subset J_2$. Up to rescaling, we can further assume $0 \le \mu$. Recall $\widetilde{\CJ} =\widetilde{\CL}_{J_1, J_2}$ is closed. Then we obtain that $(0, \gamma) \in \widetilde{\CL}_{J_1, J_2}$, which is a contradiction to the definition of $\widetilde{\CL}_{J_1, J_2}$. This finishes the claim.
\end{proof}

\subsection{The GIT quotient}

Take $\lambda\in X^+$, such that $\langle \alpha_i^\vee,\lambda\rangle>0$, for all $i\in\I$. Then $\overline{\lambda}=\lambda-\theta_X(\lambda)$ belongs to $\iwtp$ by Lemma \ref{le:bar}. We choose the ample line bundle of $\widetilde{V}_k$ as the trivial one and the $\overline{T}_k$-linearization is twisted by the character $\overline{\lambda}$. Then the geometric invariant theory (GIT) quotient (see, for example, \cite{GIT}) is
\begin{equation}\label{eq:can}
    \widetilde{V}_k /\!\!/_{\overline{\lambda}}\overline{T}_k=\text{Proj }\bigoplus_{n\ge 0}k[G_k/K_k]_{\preceq n\overline{\lambda}} e^{n \overline{\lambda}}.
\end{equation}

 For any subset $J\subset\Ir$, $\widetilde{\CL}_{\Ir, J}$ is always essential by Definition~\ref{def:EF}. Hence by Theorem~\ref{thm:CJ}, we have the associated $\widetilde{G}_k$-invariant prime ideal
 \[
 \mathfrak{p}_{\Ir, J} = \bigcup_{(\mu,\lambda) \in \widetilde{\CL}_{\Ir, J}} k[\widetilde{G}_k/\widetilde{K}_k]_{\le (\mu,\lambda)}.
 \]

We shall simply write $ \mathfrak{p}_J = \mathfrak{p}_{\Ir, J}$. Let $\widetilde{\mathcal{O}}_J\subset\widetilde{V}_k$ be the open $\widetilde{G}_k$-orbit on the closed subvariety defined by the prime ideal $\mathfrak{p}_J$. Let 
\begin{equation*}
    \widetilde{V}^0_k=\bigcup_{J\subset\Ir}\widetilde{\mathcal{O}}_J
\end{equation*}
be the union of these orbits. It is open in $\widetilde{V}_k$ by Theorem~\ref{thm:CJ}.

\begin{theorem}
    When linearized by $\overline{\lambda}$, the $\overline{T}_k$-action on $\widetilde{V}_k $ has semistable and stable locus $\widetilde{V}_k^0$. Therefore the GIT quotient $\widetilde{V}_k/\!\!/_{\overline{\lambda}}\overline{T}_k$ is the same as the geometric quotient $\widetilde{V}^0_k/\overline{T}_k$, and is independent of the choice of $\lambda$.
\end{theorem}

\begin{proof}
Take $x\in\widetilde{V}_k \backslash\widetilde{V}_k^0$. Let $\mathfrak{p}$ be the ideal of $k[\widetilde{V}_k]$ consisting of the functions which vanish on the orbit closure $\overline{\widetilde{G}_k\cdot x}$. Then $\mathfrak{p}$ is a $\widetilde{G}_k$-stable prime ideal, and $\mathfrak{p} = \mathfrak{p}_{J_1, J_2}$ for some essential $(J_1, J_2)$. We have $J_1 \neq \Ir$ by definition. It follows that $  k[G_k/K_k]_{\preceq n\overline{\lambda}}e^{n\overline{\lambda}}\subset \mathfrak{p}$, for any $n>0$. This show that $x$ is unstable. 

Note that for any $\widetilde{G}_k$-stable essential prime ideal  $\mathfrak{p}_{J_1, J_2}$, we have $  k[G_k/K_k]_{\preceq n\overline{\lambda}}e^{n\overline{\lambda}}\not\subset \mathfrak{p}_{J_1, J_2}$ for $n >0$ if and only if $J_1 = \Ir$. This shows   $\widetilde{V}_k^0$ is the semistable locus. 

To show semistable locus coincides with the stable locus it would suffice to show the $\overline{T}_k$-action on $\widetilde{V}^0_k$ is free. Let $x \in \widetilde{V}^0_k$, and denote by $N$ the stabilizer of $x$ in $\overline{T}_k$. Then $N$ fixes the closure $\overline{\widetilde{G}_k x}$, hence the coordinate ring $k[\overline{\widetilde{G}_k x}]$. We know $ k[\overline{\widetilde{G}_k x}] = k[\widetilde{V_k}]/ \mathfrak{p}_{J_1,J_2}$ for some essential $\mathfrak{p}_{J_1, J_2}$. We have $J_1 = \Ir$ by definition. Hence 
$\chi_{\mu} e^\mu \not \in \mathfrak{p}_{J_1, J_2}$ for any $\mu \in \breve{X}^+$. Since $\chi_{\mu} e^\mu$ is $N$-invariant for any $\mu \in \breve{X}^+$, $N$ must be trivial. 
\end{proof}

Let us write $P_k=\widetilde{V}^0_k/\overline{T}_k$.  Then $P_k$ is an embedding of $G_k/K_k$.  For any $J\subset\Ir$, we set $\mathcal{O}_J$ to be the image of $\widetilde{\mathcal{O}}_J$ under the canonical projection $\widetilde{V}_k^0\rightarrow P_k$. For $i\in\Ir$, we write $D_i=\overline{\mathcal{O}_{\Ir-\{i\}}}$ to be the $G_k$-stable prime divisor of $P_k$.

\begin{thm}\label{thm:sp}
    The embedding $P_k$ of $G_k/K_k$ satisfies the following properties:

    (1) $P_k$ has the unique closed $G_k$-orbit $\mathcal{O}_{\emptyset}$;

    (2) for a $B_k$-stable prime divisor $D$ of $G_k/K_k$, its closure $\overline{D}$ inside $P_k$ does not contain $\mathcal{O}_\emptyset$;

    (3) for $i\in\Ir$, the $G_k$-invariant valuation $v_i=v_{D_i}$ associated to the divisor $D_i$ satisfies: $v_i(\chi_{\overline{\alpha_i}})<0$, and $v_i(\chi_{\overline{\alpha_j}})=0$ for $j\in\Ir$ with $j\neq i$.

    Therefore $P_k$ is the (unique) canonical embedding of $G_k/K_k$.
\end{thm}

\begin{proof}
We show Part (1). By the construction, $G_k$-orbits on $P_k$ are of the form $\mathcal{O}_J$, for $J\subset\Ir$. Moreover one has $\mathcal{O}_{J'}\subset\overline{\mathcal{O}_{J''}}$ if and only if $J'\subset J''$. Hence part (1) immediately follows. 

We show Part (2). Let $\widetilde{D}$ be the closure of $D\times \overline{T}_k$ in $\widetilde{V}_k$. Then it is clear that $\widetilde{D}\cap\widetilde{V}_k^0$ is the preimage of $\overline{D}$ (in $P_k$) under the canonical map $\widetilde{V}_k^0\rightarrow P_k$. To prove (2), it suffices to show that $\widetilde{D}$ does not contain $\overline{\widetilde{\mathcal{O}}_{\emptyset}}$ in $\widetilde{V}_k$. The defining ideal $I(D)$ of $D$ in $k[G_k/K_k]$ is $B_k$-stable, so we can take $\chi_\mu\in I(D)^{U_k}$, for some $\mu\in\iwtp$. Therefore as a function on $\widetilde{V}_k$, the function $\chi_\mu e^\mu$ vanishes on $\widetilde{D}$. However by definition $\chi_\mu e^\mu$ does not belong to the ideal $\mathfrak{p}_{\emptyset}$, which is the defining ideal of $\overline{\mathcal{O}_{\emptyset}}$. This completes the proof of (2).

We show Part (3). For any $i\in\Ir$, let $\widetilde{v_i}$ be the valuation on $k(\widetilde{G}_k/\widetilde{K}_k)$ associate with the orbit closure $\overline{\widetilde{\mathcal{O}}_{\Ir-\{i\}}}$, which is defined by the prime ideal $\mathfrak{p}_{\Ir-\{i\}}$. Then for any $f\in k(G_k/K_k)$, it is clear that $v_i(f)=\widetilde{v_i}(fe^0)$. Take $\mu\in\iwtp$, such that $\mu+\overline{\alpha_j}\in\iwtp$, for any $j\in\Ir$. Then we have
\begin{equation}\label{eq:vic}
v_i(\chi_{\overline{\alpha_j}})=\widetilde{v_i}(\chi_{\overline{\alpha_j}}e^0)=\widetilde{v_i}(\chi_{\mu+\overline{\alpha_j}}e^{\mu+\overline{\alpha_j}})-\widetilde{v_i}(\chi_\mu e^{\mu+\overline{\alpha_j}}).
\end{equation}
Note that $\chi_{\mu+\overline{\alpha_j}}e^{\mu+\overline{\alpha_j}}$ belongs to $k[\widetilde{V}_k]-\mathfrak{p}_i$, so we have $\widetilde{v_i}(\chi_{\mu+\overline{\alpha_j}}e^{\mu+\overline{\alpha_j}})=0$. If $j\neq i$, we have $\chi_\mu e^{\mu+\overline{\alpha_j}}\in k[\widetilde{V}_k]-\mathfrak{p}_i$. And we have $\chi_\mu e^{\mu+\overline{\alpha_i}}\in\mathfrak{p}_i$. Therefore we have $\widetilde{v_i}(\chi_\mu e^{\mu+\overline{\alpha_j}})=0$ if $j\neq i$, and $\widetilde{v_i}(\chi_\mu e^{\mu+\overline{\alpha_j}})>0$ if $j=i$. Combined with \eqref{eq:vic}, this completes the proof of (3).
\end{proof}

\begin{remark}
    In terms of the spherical theory, Theorem \ref{thm:sp} shows that $P_k$ is the simple embedding associated to the colored cone $(\mathcal{V}(G_k/K_k),\emptyset)$. Therefore it coincides with the notion of canonical embeddings in the theory of spherical varieties (see \cite{Pe10}*{Definition 3.2.1}). 
\end{remark}

It follows from the definition that all the $G_k$-orbit closures on $P_k$ are parametrized by subsets of $\Ir$, and they are obtained by partial intersections of $G_k$-stable prime divisors.

\subsection{The integral model} We next define an integral model for the canonical embedding. For $\mu\in\iwt$, we define
\begin{equation*}
    \mathbf{O}(G/K)_{\preceq\mu}=\bigcup_{\lambda\in\iwtp,\;\lambda\preceq\mu}\mathbf{O}(G/K)_{\le\lambda}.
\end{equation*}
\begin{defi}\label{defi:P}
Let $\lambda\in X^+$ be such that $\langle \alpha_i^\vee,\lambda\rangle>0$, for all $i\in\I$. We define
\begin{equation}\label{eq:intZ}
  \mathbf{P}_\lambda(G/K)=\text{Proj }\bigoplus_{n\geq 0}\mathbf{O}(G/K)_{\preceq n\overline{\lambda}}e^{n\overline{\lambda}}
\end{equation}
to be the projective scheme over $\mathbb{Z}$. Here the multiplication on the graded ring is defined to be the obvious one. 
\end{defi}

Then for any algebraically closed field $k$ with characteristic not 2, by \eqref{eq:can} and Theorem \ref{thm:int}, the base change $Spec\; k\times_{Spec\;\mathbb{Z}}\mathbf{P}_\lambda(G/K)$ is isomorphic to the canonical embedding $P_k$. 

When $G_k$ is adjoint, the canonical embedding of $G_k/K_k$ coincides with the wonderful compactification constructed by De Concini--Procesi--Springer. Therefore by \eqref{eq:intZ} we obtain an integral model for the wonderful compactification $P_k$. One can similarly obtain the integral model for the closure of $G_k$-orbits in $P_k$.

\begin{remark}
It is clear the canonical embedding $P_k$ is independent of the choice of $\lambda$. We expect the scheme  $\mathbf{P}_\lambda(G/K)$ is also independent of the choice of $\lambda$ over $\BZ[\frac{1}{2}]$.
\end{remark}

\subsection{Local Structure Theorem and Smoothness}

In this section, we study the smoothness of the canonical embedding. We expect this is known to experts.

Recall that $G_k$ is semisimple and $P_k\supset G_k/K_k$ is the canonical embedding. Analogous to the notations in Section \ref{sec:wond}, let $\mathring{P}_k\subset P_k$ be the complementary of all the $B_k$-stable divisors which are not $G_k$-stable. Recall $\overline{T_k}=T_k/T_k^{\theta_k}$ is the torus contained in $G_k/K_k$. Let $N_k$ be the closure of $\overline{T_k}$ in $\mathring{P}_k$. Let $P_{\I_\bullet,k} \supset B_k$ be the parabolic subgroup associated with the subset of simple roots $\I_\bullet$, and let $U_{P_{\I_\bullet,k}}$ be the unipotent radical of $P_{\I_\bullet,k}$. We extend Theorem \ref{thm:smooad} to the canonical embeddings.

\begin{prop}\label{thm:smoo}
The action map
    \[
    U_{P_{\I_\bullet},k}\times N_k\longrightarrow \mathring{P_k}
    \]
    is an isomorphism of varieties. And $N_k\supset \overline{T_k}$ is an affine toric variety whose coordinate ring is the monoid algebra of $C=\{\mu\in\iwt\mid \mu\preceq 0\}\subset \iwt$. 
\end{prop}

\begin{proof}
    Take $\lambda\in X^+$, such that $\langle \alpha_i^\vee,\lambda\rangle>0$, for all $i\in\I$.  We have $P_k = \text{Proj }\bigoplus_{n\ge 0}k[G_k/K_k]_{\preceq n\overline{\lambda}} e^{n \overline{\lambda}}$. We write $ S_{\lambda} = \bigoplus_{n\ge 0}k[G_k/K_k]_{\preceq n\overline{\lambda}} e^{n \overline{\lambda}}$. By the proof of Theorem~\ref{thm:sp} (2), $\mathring{P}_k\subset{P_k}$ is the open affine subset defined by $\chi_{\overline{\lambda}}e^{\overline{\lambda}}\neq 0$.  Hence the coordinate ring $k[\mathring{P}_k]$ can be identified with the algebra 
    \[ (S_\lambda)_{(\chi_{\overline{\lambda}}e^{\overline{\lambda}})} = \left\{\frac{fe^{n\overline{\lambda}}}{\chi^n_{\overline{\lambda}}e^{n\overline{\lambda}}} \mid  f \in k[G_k/K_k]_{\preceq n\overline{\lambda}} \right\}.
    \]
   Note that 
    \begin{equation}\label{eq:sub} (S_\lambda)_{(\chi_{\overline{\lambda}}e^{\overline{\lambda}})} \cong k[G_k/K_k]_{(\chi_{\overline{\lambda}})}=\{f/\chi_{\overline{\lambda}}^n\mid f\in k[G_k/K_k]_{\preceq n\overline{\lambda}} \text{ for }n\in\mathbb{N}\}\subset k(G_k/K_k).
    \end{equation}
Hence it suffices to show the coaction map $\Delta:k[G_k/K_k]\rightarrow k[G_k]\otimes k[G_k/K_k]$ induces an injective map
\begin{equation}\label{core}
    \delta:k[G_k/K_k]_{(\chi_{\overline{\lambda}})}{\longrightarrow} k[U_{P_{\I_\bullet,k}}]\otimes k[\overline{T_k}] 
\end{equation}
and the image is $k[U_{P_{\I_\bullet,k}}]\otimes k[C]$. Here $\delta$ is the composition of $\Delta$ with the tensor product of maps $k[G_k]\rightarrow k[U_{P_{\I_\bullet,k}}]$ and $k[G_k/K_k]\rightarrow k[\overline{T_k}]$ given by restrictions. 

Recall the action map 
\begin{equation}\label{eq:act}
U_{P_{\I_\bullet},k}\times \overline{T_k}\rightarrow G_k/K_k
\end{equation}
is an isomorphism onto the open $B_k$-orbit of $G_k/K_k$ (see, for example, \cite{DS}*{Proposition~3.8}). It follows that \eqref{core} is injective.

We next determine the image. Take $f\in k[G_k/K_k]_{\le\mu}$. It follows from the proof of \cite{Lu09}*{Lemma 1.8} that $\Delta(f)\in \bigcup_{\lambda'+\lambda''=\mu}k[G_k]_{\le \lambda'}\otimes k[G_k/K_k]_{\le\lambda''}.$ On the other hand, for $f\in k[G_k]_{\le\mu}$ where $\mu\in X$, we have $f\mid_{T_k}=\sum_{\lambda}f_\lambda\chi_\lambda$ where $f_\lambda=f(\mathrm{1}_\lambda)$. Here we are viewing $f$ as a linear function on the modified quantum group as in \cite{Lu09}*{Theorem 4.11}. In particular we have $f_\lambda=0$ unless $\lambda\le\mu$. Therefore for $f\in k[G_k/K_k]_{\le\mu}$ we have $\delta(f)\in k[U_{P_{\I_\bullet},k}]\otimes k[\overline{T_k}]_{\le\mu}$, where $k[\overline{T_k}]_{\le\mu}\subset k[\overline{T_k}]$ is the subspace consisting of linear combinations of $\chi_{\mu'}$ with $\mu'\in\iwt$ and $\mu'\le\mu$. Then by the definition \eqref{eq:sub} and $\delta(\chi_{\overline{\lambda}})=1\otimes \chi _{\overline{\lambda}}$, we conclude that the image of $\delta$ lies in $k[U_{P_{\I_\bullet,k}}]\otimes k[C]$. We next show that the image is exactly $k[U_{P_{I_\bullet,k}}]\otimes k[C]$.

Let $G_{k}^{ad}=G_k/Z_k$, where $Z_k$ is the center of $G_k$. Then $\theta_k$ induces an involution on $G_{k}^{ad}$. Let $K_{k}^{ad}\subset G_{k}^{ad}$ denote the fixed-point subgroup. Let $\iwt^{ad}\subset\iwt$ be the spherical weight lattice associated with the symmetric space $G_k^{ad}/K_k^{ad}$. Take $n\in\mathbb{N}$ such that $n\overline{\lambda}\in \iwt^{ad}$. Similarly define $k[G_k^{ad}/K_k^{ad}]_{(n\chi_{\overline{\lambda}})}$ which is viewed as a subalgebra of $ k[G_k/K_k]_{(\chi_{\overline{\lambda}})}$. Moreover we have the commutative diagram
\begin{equation}
\begin{tikzcd}
    & k[G_k/K_k]_{(\chi_{\overline{\lambda}})} \arrow[r,"\delta"] &  k[U_{P_{\I_\bullet},k}]\otimes k[\overline{T_k}] \\ & k[G_k^{ad}/K_k^{ad}]_{(\chi_{n\overline{\lambda}})} \arrow[u,hook] \arrow[r,"\delta^{ad}"] & k[U_{P_{\I_\bullet},k}]\otimes k[\overline{T_k^{ad}}]\arrow[u,hook].  
\end{tikzcd}
\end{equation}
Here $\delta^{ad}$ is defined similarly as $\delta$. By Theorem \ref{thm:smooad}, the map $\delta^{ad}$ is an isomorphism into the image $k[U_{P,k}]\otimes k[C^{ad}]$, where $C^{ad}\subset \iwt^{ad}$ is defined similarly as $C$. In particular, the image of $\delta$ contains the subalgebra $k[U_{P_{\I_\bullet},k}]\subset k[U_{P_{\I_\bullet},k}]\otimes k[\overline{T_k}]$. On the other hand, take any $\mu\in C$. We can take $s\in\mathbb{Z}$ to be big enough such that $s\overline{\lambda}+\mu\in X^+$. Hence $\chi_\mu=\chi_{\mu+s\overline{\lambda}}/\chi_{\overline{\lambda}}^s$ belongs to $k[G_k/K_k]_{(\chi_{\overline{\lambda}})}$. Therefore $1\otimes \chi_\mu=\delta(\chi_\mu)$ belongs to the image of $\delta$. We complete the proof.
\end{proof}

Recall the spherical roots $\{\alpha_i'\mid i\in\Ir\}$ in Section \ref{subsec:spherical}.

\begin{cor}
    The canonical embedding $P_k$ is smooth if and only if $\{\alpha_i'\mid i\in\Ir\}$ is a basis of $\iwt$.
\end{cor}

\begin{proof}
    Since $P_k$ has no colors, any $G_k$-orbit of $P_k$ intersects $\mathring{P_k}$. Hence $P_k$ is smooth if and only if $\mathring{P_k}$ is smooth. By Proposition \ref{thm:smoo}, $\mathring{P_k}$ is smooth if and only if the toric variety $N_k$ is smooth. By the definition of the spherical root system, $\{-\alpha_i'\mid i\in\Ir\}$ forms a minimal generators of the polyhedral cone $\mathbb{R}\otimes C$. Then the corollary follows from the smoothness criteria of toric varieties.
\end{proof}

\begin{remark}
    It is clear that when $G_k$ is of adjoint type, $P_k$ is smooth. There are other cases when $P_k$ is also smooth, for example, the symmetric space $SL_2/SO_2$.
\end{remark}

\end{document}